\documentclass{amsart}
\usepackage{tikz-cd,tikz}
\usepackage{graphicx, amsmath,amsthm} 

\title{Reflexive dg categories in algebra and topology}
\author[Booth, Goodbody, Opper]{Matt Booth, Isambard Goodbody, Sebastian Opper}

\address{\mbox{Department of Mathematics,
Imperial College London,
SW7 2AZ, UK}\newline
\indent Heilbronn Institute for Mathematical Research,
Bristol, BS8 1UG, UK}
\email{matt.booth@imperial.ac.uk}
\address{School of Mathematics and Statistics,
University of Glasgow, 
G12 8QQ, UK}
\email{isambard.goodbody@glasgow.ac.uk}
\address{Department of Algebra,
Charles University,
Sokolovská 49/83, Praha 8, Czechia}
\email{opper@karlin.mff.cuni.cz }

\usepackage{amsmath, mathtools, comment, stmaryrd}
\usepackage{amsfonts}
\usepackage{natbib}
\usepackage{amssymb}
\usepackage{amsthm}
\usepackage{tikz-cd}
\usetikzlibrary{decorations, decorations.pathmorphing,decorations.pathreplacing, calc}
\usetikzlibrary{hobby}
\usepackage{hyperref}
\usepackage{cleveref}
\usepackage{url}

\tikzset{
    clip even odd rule/.code={\pgfseteorule}, 
    invclip/.style={
        clip,insert path=
            [clip even odd rule]{
                [reset cm](-\maxdimen,-\maxdimen)rectangle(\maxdimen,\maxdimen)
            }
    }
}

\newcommand{\R}{{\mathrm{\normalfont\mathbb{R}}}}
\let\hom\relax\newcommand{\hom}{\mathrm{Hom}}

\DeclareMathOperator{\perf}{perf}
\newcommand{\pvd}{{\ensuremath{{\mathcal{D}_{\mathrm{fd}}}}}\kern 1pt}
\newcommand{\per}{{\ensuremath{{\mathcal{D}^{\perf}}}}\kern 1pt}
\newcommand{\thick}{{\ensuremath{\mathbf{thick}}}}

\newcommand{\dco}{\ensuremath{\mathcal{D}^\mathrm{co}}}
\newcommand{\hqf}{{\ensuremath{\mathbf{hqf}}}}
\newcommand{\fd}{{\ensuremath{\mathbf{fd}}}}
\newcommand{\rf}{{\ensuremath{\mathbf{ref}}}\kern 1pt}

\DeclareMathOperator{\ev}{ev}
\DeclareMathOperator{\coev}{coev}

\DeclareMathOperator{\Hom}{Hom}
\DeclareMathOperator{\rad}{rad}
\DeclareMathOperator{\Ext}{Ext}

\newcommand{\spec}{\operatorname{Spec}}

\newcommand{\into}{\hookrightarrow}

\newcommand{\Flow}{\operatorname{Flow}}
\newcommand{\Ob}[1]{\operatorname{Ob}(#1)}
\newcommand{\Fuk}{\operatorname{Fuk}}

\newcommand{\cA}{\mathcal{A}}
\newcommand{\cB}{\mathcal{B}}
\newcommand{\cF}{\mathcal{F}}
\newcommand{\cM}{\mathcal{M}}

\newcommand{\cT}{\mathcal{T}}

\newtheorem{thm}{Theorem}[subsection]
\newtheorem{Introthm}{Theorem}

\newtheorem{lem}[thm]{Lemma}
\newtheorem{prop}[thm]{Proposition}
\newtheorem{cor}[thm]{Corollary}

\theoremstyle{remark}
\newtheorem{ex}[thm]{Example}
\newtheorem{rmk}[thm]{Remark}
\newtheorem{defn}[thm]{Definition}

\newcommand{\op}{\ensuremath \mathrm{op}}

\newcommand{\cD}{\mathcal{D}}
\newcommand{\cC}{\mathcal{C}}

\newcommand{\changelocaltocdepth}[1]{%
  \addtocontents{toc}{\protect\setcounter{tocdepth}{#1}}%
  \setcounter{tocdepth}{#1}%
}

\begin{document}


\begin{abstract}
Reflexive dg categories were introduced by Kuznetsov and Shinder to abstract the duality between bounded and perfect derived categories. In particular this duality relates their Hochschild cohomologies, autoequivalence groups, and semiorthogonal decompositions. We establish reflexivity in a variety of settings including {chain and cochain dg algebras of topological spaces, Ginzburg dg algebras, affine schemes, simple-minded collections, and Fukaya categories of cotangent bundles and surfaces as well as the closely related class of graded gentle algebras.} Our proofs are based on the interplay of reflexivity with gluings, derived completions, and Koszul duality. In particular we show that for certain (co)connective dg algebras, reflexivity is equivalent to derived completeness.
\end{abstract}
\maketitle
\changelocaltocdepth{1}
\tableofcontents
\section{Introduction}
Duality theorems are abundant in algebra and geometry. Reflexivity - recently introduced in \cite{KS} and with a precursor in \cite{bznp} - generalises the duality between bounded and perfect derived categories seen in \cite{Ballard, bznp, Chen}. Given a dg category $\cC$ over a field $k$, its \textbf{perfectly valued dg category} $\pvd(\cC)$ consists of the dg modules over $\cC$ which take values in the category of perfect complexes over $k$. For finite dimensional algebras and proper schemes, this construction permutes their perfect and bounded derived categories. The functor $\cC \mapsto \pvd(\cC)$ is contravariant, and a dg category $\cC$ is called \textbf{reflexive} if the natural functor $$\cC \rightarrow \pvd\big( \pvd(\cC)\big)$$ is a Morita equivalence. In \cite{KS} it was shown that for a reflexive dg category $\mathcal{C}$, there is an isomorphism between the triangulated autoequivalence groups of $\mathcal{D}^{\perf}(\mathcal{C})$ and of $\pvd(\mathcal{C})$ and a bijection between semiorthogonal decompositions. In \cite{Greflex}, reflexive dg categories are shown to be exactly the reflexive objects in the Morita homotopy category. As a consequence, the Hochschild cohomologies and derived Picard groups of $\mathcal{D}^{\perf}(\mathcal{C})$ and of $\pvd(\mathcal{C})$ are shown to coincide. Note that the analogous statement for Hochschild \textit{homology} is false \cite[Example 5.6]{Greflex}. The main families of dg categories currently known to be reflexive are
    \begin{itemize}
		\item the (perfect or bounded) derived categories of proper schemes;
		\item proper connective dg algebras;
		\item homologically smooth and proper dg categories;
        \item Fukaya categories of Milnor fibres of certain hypersurface singularities.
	\end{itemize}

    {In this paper, we give a very general `2-out-of-3' characterisation of reflexive dg algebras, and apply this criterion to a range of examples from representation theory, algebraic topology, and symplectic \& algebraic geometry. As a consequence, we extend the above list across various domains.} Important is that we work almost always over \textit{arbitrary} fields $k$, unlike the earlier \cite{bznp} and \cite{KS} who work over characteristic zero and perfect fields respectively.

	\subsection*{Topological spaces} \ 
	Associated to a topological space $X$ are two natural dg-$k$-algebras: the algebra of cochains $C^\bullet(X,k)$ and the algebra of chains on loops $C_\bullet(\Omega X,k)$. The following is a simplified version of the results of \Cref{Section: reflexivity for chains and cochains}:
	
\begin{Introthm}\label{Intro thm: reflexivity chains cochains}
Let $X$ be a path connected topological space and $k$ a field. Suppose that $H_i(\Omega X,k)$ is finite dimensional for all $i \geq 0$. Then both $C^\bullet(X,k)$ and $C_\bullet(\Omega X,k)$ are reflexive as long as either of the following conditions hold:
\begin{enumerate}
	\item $X$ is simply connected.
	\item $\pi_1(X)$ is a finite $p$-group, where $p=\operatorname{char} (k)> 0$. 
\end{enumerate}
Moreover, there are equivalences $$\begin{array}{ccc}\pvd(C^{\bullet}(X,k))\simeq \per(C_{\bullet} (\Omega X,k)) & \text{} & \per(C^{\bullet}(X,k))\simeq \pvd(C_{\bullet}( \Omega X,k)).\end{array}$$
\end{Introthm}We remark that the assumptions of the theorem are satisfied by finite simply connected CW complexes (in fact, $C^\bullet(X,k)$ is reflexive whenever $\oplus_{i\in \mathbb{N}}H_i(X,k)$ is finite dimensional) and classifying spaces of finite $p$-groups (where we moreover have $C_\bullet(\Omega BG,k) \simeq kG$). We give further examples from string topology, rational homotopy theory, $p$-compact groups, and symplectic geometry. Note that $\pvd(C_{\bullet}(\Omega X,k))$ can be interpreted as the category of $\infty$-local systems on $X$ with finite fibres. {For a general path connected $X$, we show that $\pvd(C^{\bullet}(X, k))$ is a derived completion of $C_{\bullet}(\Omega X, k)$ in the sense of \cite{efimovcompletion}, cf.~\Cref{dfdcochainslem}.}

	\subsection*{Ginzburg dg algebras and Calabi--Yau completions} \ 

 Ginzburg dg algebras are a class of Calabi--Yau dg algebras first introduced in \cite{ginzburg} and constructed from a quiver with superpotential. Under certain assumptions, \textit{all} complete exact Calabi--Yau dg algebras are Ginzburg dg algebras \cite{vdbsup}. In \cite{kellercyc}, Ginzburg dg algebras were interpreted as deformed $3$-Calabi--Yau completions. We prove the following (cf.\ \Cref{cycompref}, \Cref{ginzref}):

\begin{Introthm}\label{intro thm: ginzburgs}
Let $Q$ be a finite quiver and $k$ a field of characteristic zero.
\begin{enumerate}
    \item For all $n \geq 2$, the (completed, undeformed) $n$-Calabi--Yau completion $\hat{\Pi}_n(Q)$ is reflexive.
    \item Let $W$ be a superpotential on $Q$ such that the cycles appearing in $W$ have length at least $3$. Then the completed Ginzburg dg algebra $\hat{\Gamma}(Q,W)$ associated to the pair $(Q, W)$ is reflexive.
\end{enumerate} 
\end{Introthm}

	\subsection*{Fukaya categories of surfaces and graded gentle algebras} \ 

     We study the reflexivity of topological Fukaya categories of surfaces in the sense of \cite{HaidenKatzarkovKontsevich}. These categories admit formal generators whose endomorphism dg algebras are graded versions of \textbf{gentle algebras} \cite{AssemSkowronski}. The gentle algebras which arise in this way are smooth, but not necessarily proper. The Koszul dual perspective, that of \textit{finite dimensional} graded gentle algebras, has been studied by many authors \cite{OpperPlamondonSchroll, OpperDerivedInvariants, LekiliPolishchuk, AmiotPlamondonSchroll}.
	
	\begin{Introthm}\label{Intro thm: partially wrapped} {Let $\Sigma$ be a graded marked surface, $\Fuk(\Sigma)$ its partially wrapped Fukaya category, and $\Fuk^{\operatorname{inf}}(\Sigma)$ its infinitesimal Fukaya category, cf.~\Cref{def: Fukaya category arc system}. If $\Sigma$ contains at least one stop, and contains no boundary component without stops and with vanishing winding number, then $\Fuk(\Sigma)$ is reflexive and we moreover have equivalences $\Fuk(\Sigma) \simeq \pvd(\Fuk^{\operatorname{inf}}(\Sigma))$ and $\Fuk^{\operatorname{inf}}(\Sigma)\simeq \pvd(\Fuk(\Sigma))$.}
	\end{Introthm} 
We believe that this result is sharp (cf.~\Cref{exa: non reflexive Fukaya}). {The condition on $\Sigma$ in \Cref{Intro thm: partially wrapped} which ensures that $\Fuk(\Sigma)$ is reflexive is equivalent to the condition that any (and hence all) smooth graded gentle algebra which arises from one of its generators is finite dimensional in each degree.} {The last assertion of \Cref{Intro thm: partially wrapped} confirms an expectation expressed in \cite[Remark 3.21]{LekiliPolishchuk} and its preceding paragraph.} \Cref{Intro thm: partially wrapped} is a consequence of the following statement:
\begin{Introthm}\label{Intro thm: finite dimensional graded gentle}
Let $A$ be a finite dimensional graded gentle algebra. Then $A$ is reflexive and $\pvd(A)$ is generated as a thick subcategory by the simple $A$-modules, i.e.\  $\thick_A \left( A/\rad(A)\right)=\pvd(A)$. {In particular, if $\Sigma$ is a graded marked surface which contains at least one stop, then {the derived completion of $\Fuk(\Sigma)$} is reflexive.}  
\end{Introthm}
 The results of \Cref{Intro thm: partially wrapped} and \Cref{Intro thm: finite dimensional graded gentle} are used as ingredients for the third author's description of derived Picard groups of partially wrapped Fukaya categories and graded gentle algebras \cite{OpperGradedGentle}. Moreover, in the newest update of \cite{OpperPlamondonSchroll}, the authors describe a geometric model for the thick closure of $A/\rad(A)$. Since the grading of $A$ can be arbitrary (in particular nonconnective), no previous techniques were able to show that this geometric model does indeed describe the whole category $\pvd(A)$.  The proof of \Cref{Intro thm: finite dimensional graded gentle} is heavily based on the following result, which describes the behaviour of reflexivity under semiorthogonal gluing:

\begin{Introthm}Let $\cT$ be a proper dg category. If $\per(\cT)$ admits a semiorthogonal decomposition into reflexive dg categories, then $\cT$ is reflexive.
\end{Introthm}
 The theorem holds true under the more general assumption that $\mathcal{T}$ is \textbf{semireflexive} (cf.~\Cref{thm:Gluing over Semiorthogonal decompositions}). To prove \Cref{Intro thm: finite dimensional graded gentle}, one exploits that $\per(A)$ admits a semiorthogonal decomposition into perfect derived categories of simpler gentle algebras whose reflexivity follows easily from our other criteria.

\subsection*{Simple-minded collections and silting objects} \

Simple-minded collections and silting objects in triangulated categories abstract simple modules and projective generators for finite dimensional algebras \cite{KoenigYang}. Algebraic triangulated categories admitting simple-minded collections or silting objects are Morita equivalent to certain (co)connective dg algebras.
\begin{Introthm}[Corollaries {\ref{thm: reflexivity via KD for conn. algs}, \ref{smoothcoconncor}, \ref{propcoconncor}, \ref{thm: reflexivity via KD for coconn. algs}}]\label{Intro thm: connective coconnective} Let $A$ be a \textbf{locally proper} dg algebra, i.e.\ $H^i(A)$ is finite dimensional for all $i \in \mathbb{Z}$.
	\begin{enumerate}
		\item If $A$ is connective then $A$ is reflexive.
        \item Suppose that $A$ is coconnective and $H^0(A)$ is a semisimple $k$-algebra. If $A$ is smooth, or proper, or $H^1(A)\cong 0$, then $A$ is reflexive.
	\end{enumerate}
\end{Introthm}
As corollaries of our more general theorems, we give some reflexivity theorems for simple-minded collections (\Cref{ex: SMCs in general}), silting objects (\Cref{siltingcor}), and relative singularity categories (\Cref{ex: rel sing cats}). We also generalise a result of \cite{MaoYangHe} on derived Picard groups of coconnective dg algebras (\Cref{MYHrmk}).
\subsection*{Affine schemes} \ 
    
We characterise all reflexive noetherian affine schemes:
	\begin{Introthm}[{\Cref{noeththm}}]\label{Intro Theorem: affine schemes} Let $R$ be a commutative noetherian $k$-algebra. Then $\per(R)$ is reflexive if and only if $R$ is a finite product of complete local $k$-algebras, each of which has a residue field which is a finite extension of $k$.
	\end{Introthm}

\subsection*{Methodology} \ 
The key technical insight of this paper can be summed up by the following slogan:
$$\textbf{Reflexivity = well-behaved }\pvd\textbf{-generators }+\textbf{ derived completeness}$$
If $A$ is a dg algebra and $M$ a {right $A$-module, we put $A^{!!}_M\coloneqq \R\mathrm{End}_{\R\mathrm{End}_A(M)^\op}(M)^\op$} and regard $A^{!!}_M$ as a derived completion of $A$ along $M$. Indeed, in algebro-geometric settings this often recovers the adic completion of $A$ at an ideal \cite{DGI}. To make the above slogan precise, we prove the following `two-out-of-three' theorem:
\begin{Introthm}[\Cref{lem: two out of three}]
    Let $A$ be a dg algebra and $M$ a thick generator for $\pvd(A)$. If any two of the following hold then so does the third:
\begin{enumerate}
    \item $A$ is reflexive.
    \item $M$ is a thick generator for {$\pvd(\R\mathrm{End}_A(M)^\op)$}.
    \item The derived completion map $A \to A^{!!}_M$ is a quasi-isomorphism.
\end{enumerate}
\end{Introthm}
It is hence important for us to identify thick generators for $\pvd(A)$. For connective dg algebras one can do this using the standard t-structure, and for coconnective dg algebras with semisimple $H^0$ one can do this using the co-t-structures of \cite{nicolaskeller}. This allows us to prove the following precursor to \Cref{Intro thm: connective coconnective}:
\begin{Introthm}[\Cref{lem: coconnective semisimple H^0 = reflexive} and \Cref{connectiveDC}]\label{Intro thm: preCC}
Let $A$ be a dg algebra with $H^0(A)$ finite dimensional. Suppose that either of the following conditions hold:
\begin{itemize}
\item $A$ is connective.
\item $A$ is coconnective and $H^0(A)$ is semisimple.
\end{itemize}
Then $A$ is reflexive if and only if it is derived complete at $H^0(A)/\rad H^0(A)$.
\end{Introthm}
Now we have dealt with the existence of good generators, we need a method to check when an algebra is derived complete. To do this we pass through the well-known relationship between derived completion and Koszul duality. Indeed, \Cref{Intro thm: connective coconnective} above follows by combining \Cref{Intro thm: preCC} with some derived completion results of Fushimi \cite{Fus24} which in the augmented setting were established in \cite{MBDDP} using Koszul duality. This approach to reflexivity already has an antecedent in \cite{LekiliUeda}. We also develop a notion of reflexivity for dg coalgebras, and show that a dg coalgebra $C$ is reflexive precisely when its Koszul dual dg algebra $\Omega C$ is (\Cref{mrefcog}). The interaction between reflexivity of $C$ and reflexivity of its linear dual $C^*$ is subtle, but their interaction is key to our proof of \Cref{Intro thm: reflexivity chains cochains} via the Koszul duality between the dg algebra of chains on loops and the dg coalgebra of chains \cite{rzcubes}. Koszul duality is moreover crucial to the proof of \Cref{intro thm: ginzburgs}, as one needs to identify Ginzburg dg algebras as Koszul duals of certain cyclic $A_\infty$-algebras, as in \cite{segal, kellercyc}.

\subsection*{Acknowledgements}
The authors thank Benjamin Briggs, Joe Chuang, Dan Kaplan, Andrey Lazarev, Yank{\i} Lekili, David Pauksztello, George Raptis, Greg Stevenson, and Michael Wemyss for helpful conversations. We are grateful to Marvin Plogmann for directing us to \cite{Fus24} and to Severin Barmeier for asking about the relationship between local properness of graded gentle algebras and reflexivity.

MB was supported by the Additional Funding Programme for Mathematical Sciences, delivered by EPSRC (EP/V521917/1) and the Heilbronn Institute for Mathematical Research. IG was supported by a PhD scholarship from the Carnegie Trust for the Universities of Scotland. SO was supported by the Primus grant PRIMUS/23/SCI/006. 

\subsection{Notation and conventions}

Throughout, $k$ will denote a fixed base field. A \textbf{complex} is a cochain complex of $k$-vector spaces. We will usually use cohomological indexing for complexes; one can convert between homological and cohomological indexing by the formula $M_i = M^{-i}$. The shift of a complex $M$ will be denoted by $M[1]$, so that $M[1]^i \cong M^{i+1}$. The category of complexes is closed symmetric monoidal, with product given by the usual tensor product of complexes. 

A \textbf{dg category} is a category enriched in complexes. The derived category of right modules over a dg category $\mathcal{A}$ will be denoted by $\mathcal{D}(\mathcal{A})$; it is a pretriangulated dg category. An $\mathcal{A}$-module $M$ is \textbf{proper} or \textbf{perfectly valued} if $H^\ast(M(a))$ is a finite dimensional graded vector space for each $a \in \mathcal{A}$. The derived category of perfectly valued modules over $\mathcal{A}$ will be denoted $\pvd(\mathcal{A})$; it is a pretriangulated dg category. A \textbf{perfect} module is a compact object of $\mathcal{D}(\mathcal{A})$; these form a pretriangulated dg category which we will denote by $\per(\mathcal{A})$.  A dg category $\mathcal{A}$ is \textbf{proper} if $H^\ast(\mathcal{A}(a,b))$ is finite dimensional for all $a,b \in \mathcal{A}$. This is equivalent to each representable $\mathcal{A}(a,-)$ being a proper module, or $\mathcal{A}$ itself being a proper $\mathcal{A}$-bimodule.

A \textbf{dg algebra} is a dg category with one object, i.e.\ a complex with a compatible multiplication. A dg algebra $A$ is \textbf{connective} if $H^i(A) \cong 0$ for $i > 0$ and \textbf{coconnective} if $H^i(A) \cong 0$ for $i < 0$. A dg algebra is \textbf{finite dimensional} if its underlying complex is finite dimensional. A finite dimensional dg algebra is clearly proper. A \textbf{module} over a dg algebra is a complex $M$ with an action map $A\otimes M \to M$. We will sometimes refer to modules as \textbf{dg modules} for emphasis. All modules are by default right modules.

\section{Preliminaries}

In this section we recall the definition of reflexivity, before showing that reflexivity is closely linked to the notion of derived completeness.

\subsection{Reflexivity}

\begin{defn}
A dg category $\mathcal{A}$ is \textbf{reflexive} (resp.~\textbf{semireflexive}) if the natural map 
\[
\mathrm{ev}_\mathcal{A}:\mathcal{A}^{\op} \to \pvd (\pvd (\mathcal{A})^{\op}); \quad a \mapsto (M \mapsto M(a))
\]
is a Morita equivalence (resp.~quasi-fully-faithful).
\end{defn}

\begin{rmk}
 Equivalently, one can use the related \textbf{coevaluation map}
$$\mathrm{coev}_\mathcal{A}:\mathcal{A} \to \pvd (\pvd (\mathcal{A}))$$
or its opposite $\mathrm{coev}_{\mathcal{A}^\op}$ in the above definition, as per \cite[Lemma 3.10]{KS}.
\end{rmk}

\begin{ex} \label{refexamples}

\begin{enumerate}\hfill

\item A smooth proper dg category $\mathcal{A}$ satisfies $\pvd(\mathcal{A})\simeq \per(\mathcal{A})$ and hence is reflexive (one inclusion is clear and the other is well-known, see e.g.\ \cite{KS}). This also follows from the fact that they are the dualisable objects in the closed symmetric monoidal category $\mathrm{Hmo}$ of dg categories localised at Morita equivalences.  In particular, if $\mathcal{X}$ is a smooth proper DM stack over a field of characteristic zero, then $\per(\mathcal{X})$ is reflexive \cite{dmstacks}.

\item Proper connective dg algebras are reflexive. When $k$ is perfect this was shown in \cite{KS} and in general this appears in \cite{approx}.

\item If $X$ is a proper scheme over $k$ then both $\per(X)$ and $\mathcal{D}^b_\mathrm{coh}(X)$ are reflexive. In characteristic zero this appears in \cite{bznp} (in fact, a relative version for algebraic spaces is given). In \cite{KS} this was proved for projective schemes over perfect fields. For all fields this appears in \cite{approx}.

\item In \cite{approx}, Azumaya algebras over proper schemes over any field were shown to be reflexive. 

\item In \cite{Greflex} the power series ring $k\llbracket t \rrbracket$ was shown to be reflexive. 

\item The polynomial ring $k[t]$ is not reflexive; this follows from \Cref{noeththm} below.

\item Let $f:\mathbb{C}^n \to \mathbb{C}$ be a weighted homogeneous polynomial with an isolated critical point, and let $V$ be its Milnor fibre. Associated to $V$ are its Fukaya category $\mathcal{F}\coloneqq \mathcal{F}(V)$ and its wrapped Fukaya category $\mathcal{W}\coloneqq \mathcal{W}(V)$. If a mild numerical condition is satisfied, then Lekili--Ueda show that both $\mathcal{F}$ and $\mathcal{W}$ are reflexive \cite[Theorem 6.11]{LekiliUeda}. In fact, they show that $\pvd(\mathcal{F})\simeq \mathcal{W}$ and $\pvd(\mathcal{W})\simeq \mathcal{F}$. Similar examples not fitting into the above framework are given in \cite{LiCY}.

\item Non-proper non-examples are easy to come by: e.g.\ there are no finite dimensional modules over the Weyl algebra, or algebras of graded Laurent polynomials, and so they cannot be reflexive. 

\item Proper dg categories are semireflexive \cite{KS}.

\end{enumerate}
\end{ex}

\begin{rmk}
A key feature of a reflexive dg category $\mathcal{A}$ is that there is some common information between $\pvd(\mathcal{A})$ and $\per(\mathcal{A})$:
\begin{enumerate}
\item For any dg category $\mathcal{A}$, $\pvd(\mathcal{A})$ is determined by $\per(\mathcal{A})$; reflexivity guarantees the converse.

\item In \cite{KS}, it was shown that for $\mathcal{A}$ reflexive there is a bijection between semiorthogonal decompositions of $\pvd(\mathcal{A})$ and of $\per(\mathcal{A})$, and an isomorphism between the triangulated autoequivalence groups of these categories.

\item It follows immediately from the results of \cite{Greflex} that if $\mathcal{A}$ is reflexive, then there is an isomorphism between the derived Picard groups of $\pvd(\mathcal{A})$ and $\per(\mathcal{A})$. It was also shown in  \textit{op.~cit.}\ that  $\pvd(\mathcal{A})$ and $\per(\mathcal{A})$ have the same Hochschild cohomology. 

\end{enumerate}

\end{rmk}

\begin{rmk} \
\begin{enumerate}
\item $\mathcal{A}$ is reflexive if and only if $\mathcal{A}^{\op}$ is \cite{KS}. 

\item Since $\per(\mathcal{A})$ is a Morita fibrant replacement of $\mathcal{A}$, and a Morita equivalence between pretriangulated idempotent complete dg categories is a quasi-equivalence, a dg category is reflexive if and only if the natural map
\[
\per(\mathcal{A})^{\op} \to \pvd (\pvd (\mathcal{A})^{\op}); \quad M \mapsto \R\hom_\mathcal{A}(M,-)
\]
is a quasi-equivalence. From this point of view, reflexivity is seen to be a representability property for functors defined on $\pvd(\mathcal{A})$. 

\item Let $\mathrm{Hmo}$ denote the Morita homotopy category of dg categories. It is a closed symmetric monoidal category, with monoidal structure induced by the derived tensor product of dg categories and internal hom induced by the internal hom of dg categories \cite{RodriguezGonzalez}. In \cite{Greflex} it was shown that the reflexive dg categories are precisely the reflexive objects in $\mathrm{Hmo}$. Loosely, this is because we have
$$\R\hom_{\mathrm{Hmo}}(\mathcal{A},k) \simeq \R\hom_{\mathrm{Hqe}}(\mathcal{A},\per (k)) \simeq \pvd (\mathcal{A}^{\op})$$so that the Morita dual of a dg category $\mathcal{A}$ coincides with $\pvd(\mathcal{A}^{\op})$.

\end{enumerate}
\end{rmk}

\subsection{$\pvd$-generators}\label{section: generators}
We are interested in dg categories of the form $\pvd(\mathcal{A})$. In many situations of interest these come with natural generators. For example, for finite dimensional algebras one takes the sum of all the simple modules, and for commutative local augmented $k$-algebras one takes $k$ itself. A special case of \cite[Theorem 0.15]{StronGen} shows that $\pvd(X)$ admits a generator whenever $X$ is a proper scheme. In this section we record some simple properties of these generators and give some examples.

\begin{defn}
    Let $\mathcal{A}$ be a dg category and $S\in \pvd(\mathcal{A})$ a perfectly valued $\mathcal{A}$-module. Say that $S$ is a \textbf{$\pvd$-generator} if we have an equality $\thick_{{\mathcal{D}}(\mathcal{A})}(S)=\pvd(\mathcal{A})$ of subcategories of $\mathcal{D}(A)$.
\end{defn}

When ${\mathcal{D}}(\mathcal{A})$ admits a t-structure, we can reduce the question of existence of a $\pvd$-generator for $\mathcal{A}$ to a question about a $\pvd$-generator for the heart:

\begin{prop}\label{prop: Dfd generator connective case}
Let $A$ be a connective dg algebra. Then $\pvd(A)$ is generated by the finite dimensional simple $H^0(A)$-modules. 
\end{prop}

\begin{proof}
Since $A$ is connective, the cohomology of any $A$-module can be viewed as an $H^0(A)$-module via restriction along $A \to H^0(A)$. If $X \in  \pvd(A)$ then $X \in \thick(H^\ast(X))$ using the standard $t$-structure. Each $H^i(X)$ is finite dimensional, and in particular a finite length $H^0(A)$-module, and so admits a finite filtration whose factors are finite dimensional simple $H^0(A)$-modules. It follows that each $H^i(A)$ is contained in the thick subcategory of $\mathcal{{\mathcal{D}}}(H^0(A))$ generated by the finite dimensional simples. Now $H^\ast(X)$ is in the image of the restriction functor $\pvd(H^0(A)) \to \pvd(A)$ and so $H^{\ast}(X)$ is in the thick subcategory of $A$ generated by the finite dimensional simple $H^0(A)$-modules. 
\end{proof}

\begin{cor}\label{congen}
    Let $A$ be a connective dg algebra such that $H^0(A)$ is finite dimensional. Let $S = H^0(A)/\mathrm{rad}H^0(A)$ be the maximal semisimple quotient of $H^0(A)$, and regard $S$ as an $A$-module. Then $S$ is a $\pvd$-generator for $A$.
\end{cor}
\begin{proof}
    Up to multiplicity, $S$ is the direct sum of the simple $H^0(A)$-modules. Hence $S$ generates the same thick subcategory as the simple $H^0(A)$-modules do.
\end{proof}

\begin{rmk}
    There is an analogous theorem for coconnective dg algebras with semisimple $H^0$ due to Keller and Nicol\'as \cite{nicolaskeller}. The proof makes use of weight structures, which we will return to in \Cref{weightsection}.
\end{rmk}

\subsection{Derived completion}\label{section: derived completion}
When $A$ admits a $\pvd$-generator $S$, standard tilting theorems imply that $\pvd(A)$ is Morita equivalent to the endomorphism dg algebra $\R\mathrm{End}_A(S)$. In this section we explore the properties of this construction from the viewpoint of reflexivity. We pay particular attention to the two-fold application of this construction, known as the \textbf{derived double centraliser} or the \textbf{derived completion} \cite{DGI, efimovcompletion}. We will show that in favourable circumstances, being derived complete with respect to a $\pvd$-generator is equivalent to being reflexive.

\begin{defn}
    Let $A$ be a dg algebra and $M$ an $A$-module. We define a new dg algebra $A^!_M\coloneqq \R\mathrm{End}_A(M)^\op$. We refer to $A^!_M$ as the \textbf{centraliser of $A$ relative to $M$}.
\end{defn}

We think of $A^!_M$ as an $M$-relative dual of $A$ - indeed, when $M$ is a $\pvd$-generator it is the Morita dual of $A$, and when $M$ is the base field it is (up to an opposite) the Koszul dual of $A$. We will explore this latter perspective further in \Cref{kdsection}.

Clearly $M$ is itself a right $A^!_M$-module, and so we may form the double dual $A^{!!}_M \coloneqq (A^!_M)^!_M$. We call this the \textbf{derived completion of $A$ along $M$}. Observe that there is a functor
\[
\R\hom_A(-,M)\colon \mathcal{D}(A)^{\op} \to \mathcal{D}(A^{!}_M) 
\]
which sends $A$ to $M$, which induces a map of dg algebras
\[
A = \R\hom_A(A,A) \to \R\hom_{A^{!}_M}(M,M)^{\op} = A_M^{!!}. 
\]

\begin{defn}
We say that a dg algebra $A$ is \textbf{derived complete with respect to $M \in \mathcal{D}(A)$} if the natural map $A \to A_M^{!!}$ is a quasi-isomorphism. In \cite{DGI} this is called \textbf{dc-completeness}.
\end{defn}

\begin{lem}
If $A$ is derived complete with respect to $M$ then $A^!_M$ is derived complete with respect to $M$.
\end{lem}

\begin{proof}
Consider the composition
\[
\mathcal{D}(A^!_M) \xrightarrow{\R\hom(-,M)} \mathcal{D}(A^{!!}_M)^{\op} \xrightarrow{\simeq} \mathcal{D}(A)^{\op} 
\]
The second functor is the equivalence induced by restricting along $A \to A^{!!}_M$. The long composite is a coproduct-preserving functor which is fully faithful, since it is fully faithful when restricted to the generator $M$: one has by hypothesis a natural quasi-isomorphism $\R\mathrm{End}_{A^!_M}(M)^\op\eqqcolon A^{!!}_M \simeq A$. Therefore the first map is fully faithful, and so the induced map $A_M^! \to (A_M^!)^{!!}_M$ is an equivalence. 
\end{proof}

\begin{rmk}\label{dercomprmk}
\begin{enumerate}\hfill

\item It was shown in \cite[Proposition 4.20]{DGI} that if $(R,\mathfrak{m},K)$ is a commutative noetherian local ring, then the map $R \to R^{!!}_K$ coincides with the $\mathfrak{m}$-adic completion map $R \to \hat{R}_\mathfrak{m}$. More generally, the same holds for any regular quotient of $R$. Efimov generalised this to a non-affine version \cite{efimovcompletion}.
\item If $\thick_A(M) = \thick_A(M')$, then clearly $A_M^!$ and $A_{M'}^!$ are Morita equivalent. By \cite[Proposition 3.2]{efimovcompletion}, in this situation $A_M^{!!}$ and $A_{M'}^{!!}$ are also Morita equivalent. 

\item \label{Morita invariance derived completion} By \cite[Proposition 3.4]{efimovcompletion}, derived completion respects Morita equivalences: if $A$ and $B$ are Morita equivalent dg algebras, with $M$ an $A$-module and $N$ the corresponding $B$-module, then $A_M^{!!}$ and $B_{N}^{!!}$ are Morita equivalent.
\end{enumerate}
\end{rmk}

\begin{lem}\label{thicklem}
Let $A$ be a dg algebra and $M$ an $A$-module. If $A \in \thick_A(M)$ then $A$ is derived complete with respect to $M$. 
\end{lem}

\begin{proof}
We have an equivalence
\[
\thick_A(M) \xrightarrow{\R\hom_A(-,M)} \mathcal{D}^{\perf}(A^{!}_M)^\mathrm{op}
\]
Indeed the functor is fully faithful restricted to the generator $M$ by definition of $A_M^!$. Then since $\mathcal{D}^{\perf}(A^!_M) = \thick(A_M^!)$, it is essentially surjective. Since $A \in \thick_A(M)$ it restricts to an equivalence
\[
\mathcal{D}^{\perf}(A) \xrightarrow{\R\hom_A(-,M)} \thick_{A_M^!}(M)^\mathrm{op}
\]
Indeed it is fully faithful as it is the restriction of a fully faithful functor. It is essentially surjective since $A$ maps to $M$. Therefore, by definition the map $A \to A^{!!}_M$ is a quasi-isomorphism.
\end{proof}

Let $A$ be a finite dimensional dg algebra. Orlov introduced the concept of a \textbf{dg radical} $J_-$ of $A$ \cite{orlovrad}. It follows from the results of \cite{orlovrad} that every module over a finite dimensional dg algebra $A$ whose underlying chain complex is finite dimensional is in the thick subcategory generated by $A/J_-$.

\begin{cor}\label{proper is derived complete} \hfill
\begin{enumerate}
\item Let $A$ be a proper dg algebra and $M$ a $\pvd$-generator of $A$. Then $A$ is derived complete with respect to $M$.
\item Let $A$ be a finite dimensional dg algebra. Then $A$ is derived complete with respect to $A/J_-$.
\end{enumerate}
 
\end{cor}

\begin{proof}
If $A$ is proper, then $A \in \pvd(A) = \thick_A(M)$ and if $A$ is finite dimensional then $A \in \thick(A/J_-)$. Both claims now follow from \Cref{thicklem}. 
\end{proof}

\begin{rmk}
In the general setting, the existence of a $\pvd$-generator for finite dimensional dg algebras is subtle. For example, Efimov constructed an example of a formal coconnective finite dimensional dg algebra $A$, and a module over it which has finite dimensional cohomology but which is not quasi-isomorphic to a finite dimensional module \cite{efimovhodge}.
\end{rmk}

We can now link reflexivity to derived completeness:

\begin{thm}[`two-out-of-three' theorem]\label{lem: two out of three}
Let $A$ be a dg algebra and $M$ a $\pvd$-generator for $A$. If any two of the following hold then so does the third:
\begin{enumerate}
    \item $A$ is reflexive.
    \item $M$ is a $\pvd$-generator for $A_M^!$.
    \item $A$ is derived complete at $M$.
\end{enumerate}
    
\end{thm}
\begin{proof}
Consider the commutative diagram 
\[
\begin{tikzcd}
\per (A)^{\op}  \arrow[r,"\ev"] \arrow[d] & \pvd( \pvd (A)^{\op}) \arrow[d,"\simeq"] \\
\per (A_M^{!!})^{\op} \simeq \thick_{A^!_M}(M) \arrow[r,hook] & \pvd (A^!_M) \simeq \pvd (\thick_A(M)^{\op})
\end{tikzcd}
\]

The assumption that $M$ is a $\pvd $-generator for $A$ implies that the right-hand vertical map is an isomorphism. Condition (1) is equivalent to the upper horizontal map being an equivalence, condition (2) is equivalent to the left-hand vertical map being an equivalence, and condition (3) is equivalent to the lower horizontal map being an equivalence. 
\end{proof} 
From the proof of \Cref{lem: two out of three}, one can immediately deduce the following:

\begin{cor}\label{semirmk}
Let $A$ be a dg algebra. If $A$ is derived complete with respect to a $\pvd$-generator, then $A$ is semireflexive.
\end{cor}

\begin{rmk}
In \cite{efimovcompletion}, the completion of a dg category $\mathcal{A}$ along any subcategory of $\mathcal{D}(\mathcal{A})$ is defined. In particular, one can complete $\mathcal{A}$ along ${\pvd(\mathcal{A})}$, and this completion comes with a natural map $\mathcal{A} \to \hat{\mathcal{A}}_{\pvd(\mathcal{A})}$. Although this does not agree with the natural morphism to $\pvd\pvd(A)$, c.f.\ \Cref{EfimovRemark}, one can formulate a result similar to \Cref{lem: two out of three} where $\pvd\pvd(A)$ is replaced by $\hat{\mathcal{A}}_{\pvd(\mathcal{A})}$.
\end{rmk}

\begin{prop}
If $A$ is a proper dg algebra and $M$ is a $\pvd$-generator for $A$, then $A$ is reflexive if and only if $M$ is a $\pvd$-generator for $A^!_M$. In this case $A^!_M$ is also reflexive. 
\end{prop}

\begin{proof}
The equivalence follows from \Cref{lem: two out of three} combined with \Cref{proper is derived complete}. If $A$ is reflexive then so is $\pvd(A)$, which is Morita equivalent to $A^!_M$, up to an opposite.
\end{proof}

\begin{lem}\label{lem: reflexive and Koszul dual} Let $A$ be a dg algebra and $M \in \pvd(A)$. Suppose that $A$ is derived complete with respect to $M$. Then:
\begin{enumerate}
\item If $\thick_A(M) = \pvd(A)$ and $\thick_{A_M^!}(M) = \pvd(A_M^!)$, then $A$ and $A_M^!$ are reflexive.

\item The following are equivalent:
\begin{enumerate}
    \item $A$ is reflexive and $\thick_A(M) = \pvd(A)$ .
    \item $A^!_M$ is reflexive and $\thick_{A_M^!}(M) = \pvd(A_M^!)$.
\end{enumerate}
\end{enumerate}
\end{lem}

\begin{proof}
If the generation conditions in (1) hold, then since $A$ is derived complete, \Cref{lem: two out of three} implies that $A$ is reflexive. Hence $\pvd(A) \simeq \mathcal{D}^{\perf}(A_M^!)$ is also reflexive. For (2), if $\thick_A(M) = \pvd(A)$ and $A$ is reflexive, then \Cref{lem: two out of three} implies that $\thick_{A_M^!}(M) = \pvd(A_M^!)$. Also $\pvd(A) = \per(A^!_M)$ is reflexive. Conversely, suppose that $A^!_M$ is reflexive and $\thick_{A_M^!}(M) = \pvd(A_M^!)$. Then $\pvd(A^!_M) = \thick_{A^!}(M) \simeq \mathcal{D}^{\perf}((A^{!!}_M)^{\op}) \simeq \mathcal{D}^{\perf}(A^{\op})$ is reflexive. Hence $A$ is reflexive. Since $A$ is derived complete, so is $A^!_M$. Then \Cref{lem: two out of three} applied to $A_M^!$ implies that $M$ generates $\pvd(A^{!!}_M)$. Since $A$ is derived complete, restriction induces an equivalence $\pvd(A^{!!}_M) \simeq \pvd(A)$, and so $M$ generates $\pvd(A)$, as required.
\end{proof}

\subsection{Restricted reflexivity}

Let $\mathcal{A}$ be a dg category and $\mathcal{B}$ a pretriangulated dg subcategory of $\pvd(\mathcal{A})$. Applying the $\pvd$ functor to the inclusion $\mathcal{B}^\op \into \pvd(\mathcal{A})^\op$ gives a map $\pvd(\pvd(\mathcal{A})^\op) \to \pvd(\mathcal{B}^\op)$. Composition with $\ev_\mathcal{A}$ hence yields a map
$$\ev_{\mathcal{A},\mathcal{B}}\colon\mathcal{A}^\op \to \pvd(\mathcal{B}^\op)$$which we call the \textbf{restricted evaluation functor}. We say that $\mathcal{A}$ is \textbf{$\mathcal{B}$-restricted reflexive} when $\ev_{\mathcal{A},\mathcal{B}}$ is a quasi-equivalence. When $\mathcal{B} = \thick_{A}(M)$ for some $M \in \pvd(A)$,  then we replace $\mathcal{B}$ by $M$ in the above notation.
\begin{ex}\label{restobs}
    Clearly $\mathcal{A}$ is reflexive precisely when it is $\pvd(A)$-restricted reflexive.
\end{ex}
\begin{prop}
    Let $\mathcal{A}$ be a dg category and $\mathcal{B}\subseteq \pvd(A)$ a reflexive subcategory. If any two of the following three conditions hold, then so does the third:\hfill
    \begin{enumerate}
        \item The inclusion $\mathcal{B}\into \pvd(A)$ is a quasi-equivalence.
        \item $\mathcal{A}$ is $\mathcal{B}$-restricted reflexive.
        \item $\mathcal{A}$ is reflexive.
    \end{enumerate}
\end{prop}
Before we begin the proof, observe that when $\mathcal{B}=\thick_{{\mathcal{D}}(\mathcal{A})}({M})$, then condition (1) says that $M$ is a $\pvd$-generator for $\mathcal{A}$.
\begin{proof}
We have already observed in \Cref{restobs} that if (1) holds, then (2) is equivalent to (3), even without the reflexivity hypothesis on $\mathcal{B}$. We need only show that (2) and (3) together imply (1). If (2) and (3) hold, then $\pvd(\mathcal{A})$ is also reflexive and moreover the natural map $\pvd(\pvd(\mathcal{A})^\op) \to \pvd(\mathcal{B}^\op)$ is necessarily a quasi-equivalence. We can conclude that (1) holds by applying $\pvd$ to this quasi-equivalence and using that both $\pvd(\mathcal{A})$ and $\mathcal{B}$ are reflexive.
\end{proof}

In general the relationship between reflexivity and restricted reflexivity seems unclear; neither implies the other.

\begin{rmk}
    If $\mathcal{A}$ is a proper dg category and $\mathcal{B}\subseteq \pvd(A)$ contains the image of $\mathcal{A}$ under the Yoneda embedding, then one can adapt the arguments of \cite{KS} to show that $\ev_{\mathcal{A},\mathcal{B}}$ is quasi-fully faithful (one should call this property $\mathcal{B}$\textbf{-restricted semireflexivity}).
\end{rmk}

\section{Connective and coconnective dg algebras}\label{weightsection}

Coconnective dg algebras with semisimple $H^0$ appear naturally as algebras of cochains on topological spaces, as derived endomorphism algebras of semisimple modules, and, more generally, as derived endomorphisms of simple-minded collections in the sense of \cite{KoenigYang}. In this section, we use the weight structures constructed by Keller and Nicol\'as \cite{nicolaskeller} to investigate when such dg algebras are reflexive. A similar argument also gives a criterion for reflexivity of connective dg algebras with finite dimensional $H^0$, which appear when considering silting objects.

\subsection{Weight Structures}
A weight structure (or a co-t-structure) on a triangulated category generalises the properties of the brutal truncation functors, just like t-structures generalise the properties of the good truncation functors. Weight structures were introduced independently by Pauksztello \cite{pauksztello} and Bondarko \cite{bondarko}. It was shown in \cite{nicolaskeller} that the derived category of a coconnective dg algebra with semisimple $H^0$ admits a particularly well-behaved weight structure. We begin by recalling these results.

\begin{defn}
A \textbf{weight structure} (or \textbf{co-t-structure}) $(\mathcal{T}^{w > 0}, \mathcal{T}^{w \leq 0})$  on a triangulated category $\mathcal{T}$ with shift functor $[1]$ is a pair of additive subcategories closed under summands that satisfy the following conditions:
\begin{enumerate}
\item $\mathcal{T}^{w > 0}$ is closed under $[-1]$ and $\mathcal{T}^{w \leq 0}$ is closed under $[1]$. 
\item $\Hom_{\mathcal{T}}(\mathcal{T}^{w > 0}, \mathcal{T}^{w \leq 0}) \cong 0$.

\item For every object $X \in \mathcal{T}$, there is a truncation triangle 
\[
\sigma_{>0} X \to X \to \sigma_{\leq 0}X \to 
\]
with $\sigma_{>0} X \in \mathcal{T}^{w > 0}$ and $\sigma_{\leq 0}X \in \mathcal{T}^{w \leq 0}$.
\end{enumerate}
\end{defn}

\begin{ex}
The typical example of a weight structure is the one on the homotopy category of an additive category $\mathcal{K}(\mathcal{C})$ given by $(\mathcal{K}(\mathcal{C})^{w > 0}, \mathcal{K}(\mathcal{C})^{w \leq 0})$ where $\mathcal{K}(\mathcal{C})^{w > 0}$ consists of complexes isomorphic to those concentrated in positive degrees and $\mathcal{K}(\mathcal{C})^{w \leq 0}$ consists of those concentrated in non-positive degrees. 
\end{ex}

\begin{thm}[{\cite[Corollary 5.1]{nicolaskeller}}]
Let $A$ be a coconnective dg algebra with $H^0(A)$ semisimple. Then:
\begin{enumerate}
\item There is a weight structure $(\mathcal{D}(A)^{w > 0},\mathcal{D}(A)^{w \leq 0})$ on $\mathcal{D}(A)$ given by 
\begin{align*}
\mathcal{D}(A)^{w > 0} &= \{X \in \mathcal{D}(A) \mid H^i(X) = 0 \text{ for } i \leq 0 \} \\
\mathcal{D}(A)^{w \leq 0} &= \{X \in \mathcal{D}(A) \mid H^i(X) = 0 \text{ for } i > 0 \}. 
\end{align*}

\item For every $X \in \mathcal{D}(A)$ there is a truncation triangle $\sigma_{>0} X \to X \to \sigma_{\leq 0}X \to$ such that the map $\sigma_{>0} X \to X$ induces an isomorphism on $H^i$ for $i \leq 0$ and the map $X \to \sigma_{\leq 0}X$ induces an isomorphism on $H^i$ for $i > 0$. 
\end{enumerate}
\end{thm}

\begin{rmk}
If $A$ is strictly coconnective, i.e. $A^i = 0$ for $i <0$, then we may take $\sigma_{\leq 0}A$ to be the brutal truncation $A^0$. 
\end{rmk}

The following result of Keller and Nicol\'as provides a $\pvd$-generator for our algebras of interest.

\begin{prop}[{\cite[5.6.1 a)]{nicolaskeller}}]\label{prop: keller-nicolas locality}
Let $A$ be a coconnective dg algebra such that $H^0(A)$ is a finite dimensional semisimple $k$-algebra. Then $\sigma_{\leq 0} A$ is a $\pvd$-generator for $A$.
\end{prop}
For a coconnective dg algebra $A$ with finite dimensional semisimple $H^0$ we set
\[
A^! \coloneqq A^!_{\sigma_{\leq 0}A} = \R\hom_A(\sigma_{\leq 0}A,\sigma_{\leq 0}A)^{\op}
\]
We note that by \cite[ Lemma 6.2]{nicolaskeller}, $A^!$ is a connective dg algebra. By \Cref{prop: keller-nicolas locality}, there is a quasi-equivalence $\pvd(A) \simeq \per(A^!)^\mathrm{op}$. For the remainder of this section, we will say that $A$ is \textbf{derived complete} to mean that $A$ is derived complete at $\sigma_{\leq 0}A$.

\subsection{The simple-projective bijection}

The aim of this subsection is to relate the simple $H^0(A)$ modules and the simple $H^0(A^!)$ modules. We begin with a version of the simple-projective bijection which follows from the results of \cite{nicolaskeller}.

\begin{thm} \label{simpleprojbij}
Let $A$ be a coconnective dg algebra with $H^0(A)$ a finite dimensional semisimple algebra. Then there is a bijection between the indecomposable summands of $H^0(A^!)$ and the isomorphism classes of simple $H^0(A)$-modules.
\end{thm}

\begin{proof} Put $S\coloneqq \sigma_{\leq 0} A$. Any indecomposable summand of $H^0(A^!)$ is of the form $P = e_PH^0(A^!)$ for some primitive idempotent $e_P$ of $H^0(A^!) = \Hom_{\mathcal{D}(A)}(S,S)^\mathrm{op}$. This idempotent splits in $\mathcal{D}(A)$ and produces an indecomposable summand $\overline{\phi(P)}$ of $S$ and so a summand $\phi(P)$ of $H^0(S) \simeq H^0(A)$. Suppose that $\phi(P)$ is decomposable, so that $\phi(P) \simeq \bigoplus S_i$ for some simple summands $S_i$ of $H^0(A)$. By \cite[Lemma 5.5]{nicolaskeller}, there are indecomposable summands $\tilde{A_i}$ of $A$ in $\mathcal{D}(A)$ lifting each of the $S_i$. By Lemma 5.4 \textit{op.~cit.}\, the isomorphism $\bigoplus S_i \to \phi(P)$ can be lifted to a map $f: \bigoplus A_i \to \overline{\phi(P)}$ which is an isomorphism on $H^0$. Since $\phi(P) \in \mathcal{D}(A)^{w \leq 0}$, the map $f$ factors as $\tilde{f}: \sigma_{\leq 0} (\bigoplus A_i) \to \overline{\phi(P)}$ and $H^0(\tilde{f})$ is an isomorphism. But then $\tilde{f}$ is a quasi-isomorphism and so $\overline{\phi(P)} \simeq \bigoplus \sigma_{\leq 0} A_i$ using Lemma 4.2 \textit{op.~cit.}\ . This contradicts the indecomposablility of $\overline{\phi(P)}$. Therefore $\phi(P)$ is an indecomposable summand of $H^0(A)$ and hence a simple $H^0(A)$ module. 

Given a simple $H^0(A)$ module, Corollary 5.7 \textit{op.~cit.}\ and its proof show that it can be lifted to a summand of $S$. The same argument as above shows that it must be indecomposable and so it corresponds to a primitive idempotent in $H^0(A^!)$. 

Suppose that $\phi(P') \simeq \phi(P)$. Then by the uniqueness of Corollary 5.7 \textit{op.~cit.}\ , we have that $\overline{\phi(P)} \simeq \overline{\phi(P')}$. Therefore there are isomorphisms
\[
e_P H^0(A^!) \simeq \Hom_{\mathcal{D}(A)}(S,\overline{\phi(P)}) \simeq \Hom_{\mathcal{D}(A)}(S,\overline{\phi(P')}) \simeq e_{P'} H^0(A^!)
\]
where the first isomorphism follows from the bijection between idempotents and summands for $H^0(A^!)$. 
\end{proof}

Recall that a ring $R$ is \textbf{semiperfect} if it has a complete set of orthogonal idempotents $e_i$ such that each $e_iRe_i$ is local. Local rings and Artinian rings are semiperfect. For semiperfect rings there is a bijection between indecomposable projectives and simples given by taking projective covers. \Cref{simpleprojbij} allows us to relate the simple $H^0(A)$-modules to the projective $H^0(A^!)$ modules. We will show now that $H^0(A^!)$ is semiperfect, and so we can relate the simple $H^0(A)$-modules to the simple $H^0(A^!)$-modules.

\begin{rmk}
If we assume that $H^0(A)$ is a product of division algebras, then by \Cref{semiperf} and Corollary 24.12 in \cite{LamNC}, every indecomposable projective is a summand of $H^0(A^!)$. So in this setting the statement of \Cref{simpleprojbij} can simplified to replace indecomposable summands of $H^0(A^!)$ with indecomposable projective $H^0(A^!)$-modules. In fact, up to Morita equivalence, the following lemma shows that we may always assume this. 
\end{rmk}

\begin{lem}\label{divalg}
Let $A$ be a coconnective dg algebra such that $H^0(A)$ is a semisimple algebra. Then there is a dg algebra $A'$ such that $\per(A) \simeq \per(A')$, and $H^0(A')$ is a product of division algebras.
\end{lem}

\begin{proof}
Suppose $H^0(A) \simeq \bigoplus_{i = 1,\dots,n, j= 1,\dots, k_i} S_{i,j}$ is a decomposition of $H^0(A)$ into simple right modules such that $S_{i,j} \simeq S_{i,j'}$ for $1 \leq j,j' \leq k_i$ for all $i$, and $S_{i,j}\not\simeq S_{i',j'}$ if $i\neq i'$. Then $H^0(A) \simeq M_{k_1}(D_1) \times \dots \times M_{k_n}(D_n)$ as algebras, where each $D_i\simeq \mathrm{End}_{H^0(A)}(S_{i,1})$ is a division algebra. By \cite[Lemma 5.5]{nicolaskeller}, there is an $A$-linear quasi-isomorphism $A \simeq \bigoplus A_{i,j}$ where $H^0(A_{i,j}) \simeq S_{i,j}$. For each $i$ and each $1 \leq j, j'\leq k_i$ there is an isomorphism $S_{i,j} \to S_{i,j'}$ which lifts to a map $f:A_{i,j} \to A_{i,j'}$ by Lemma 5.4 \textit{op.~cit.}\ , and moreover by naturality of Lemma 5.4 \textit{op.~cit.}\ it follows that $f$ is a quasi-isomorphism. So if $A':= \R\hom_A(\bigoplus_i A_i,\bigoplus_i A_i)$ where $A_i\coloneqq A_{i,1}$ then it follows that $\mathcal{D}^{\perf}(A) \simeq \mathcal{D}^{\perf}(A')$. Furthermore we have that $H^0(A') \simeq D_1 \times \dots \times D_n$, as required. 
\end{proof}

\begin{lem}\label{coconnlem}
Let $A$ be a coconnective dg algebra with semisimple $H^0(A)$. Then $H^0(\sigma_{\leq 0}A)$ is a semisimple $H^0(A^!)$-module.
\end{lem}

\begin{proof}
Set $S = \sigma_{\leq 0} A$. The functor $H^0$ induces a map of $k$-algebras
\[
H^0(A^!)^\mathrm{op} \cong \mathrm{End}_{\mathcal{D}(A)}(S) \to \mathrm{End}_{H^0(A)}(H^0(S)) \cong \mathrm{End}_{H^0(A)}(H^0(A)) \cong H^0(A)
\]
using the isomorphism $H^0(A) \cong H^0(S)$. Since $H^0(A)$ is a semisimple $k$-algebra, it is a semisimple $H^0(A^!)$-module, and moreover the isomorphism $H^0(A) \cong H^0(S)$ is $H^0(A)$-linear. It remains to check that this is an isomorphism of $H^0(A^!)$-modules where $H^0(A)$ is viewed as an $H^0(A^!)$ module via restriction and $H^0(S)$ is viewed as an $H^0(A^!)$ module using the fact that $H^0(A^!)^\mathrm{op}$ is the endomorphism ring of $S$ in $\mathcal{D}(A)$. To check this, note that the map $H^0(A^!)^\mathrm{op} \to H^0(A)$ can be identified with the map 
\[
H^0(A^!)^\mathrm{op} \cong \Hom_{\mathcal{D}(A)}(S,S) \xrightarrow{f^\ast} \Hom_{\mathcal{D}(A)}(A,S) \xrightarrow{f_{\ast}^{-1}} \Hom_{\mathcal{D}(A)}(A,A) \simeq H^0(A)
\]
where $f: A \to S$ and where $f_{\ast}$ agrees with the isomorphism $H^0(f)\colon H^0(A) \rightarrow H^0(S)$ across the identification $\hom_{\mathcal{D}(A)}(A,-)\simeq H^0(-)$. One can then check directly that the actions agree. 
\end{proof}

\begin{lem} \label{semiperf}
Let $A$ be a coconnective dg algebra such that $H^0(A)$ is a product of division algebras. Then $H^0(A^!)$ is semiperfect.
\end{lem}

\begin{proof}
Suppose $H^0(A) = D_1 \times \dots \times D_n$ is a product of division algebras. By \cite[ Lemma 5.5]{nicolaskeller}, $A$ splits into a direct sum of indecomposables $A \simeq \bigoplus A_i$ such that $H^0(A_i) \simeq D_i$. It follows that $S \coloneqq \sigma _{\leq 0} A \simeq \bigoplus S_i$ where we put $S_i\coloneqq \sigma_{\leq 0}A_i$. Therefore $H^0(A^!)^\mathrm{op} \simeq \Hom_{\mathcal{D}(A)}(\bigoplus S_i, \bigoplus S_i)$. The maps $e_i: S_i \to S \to S_i$ clearly form a complete set of orthogonal idempotents so it remains to show that each $e_i H^0(A)^\mathrm{op} e_i = \Hom_{\mathcal{D}(A^!)}(S_i,S_i)$ is local. Consider the map of algebras from the proof of \Cref{coconnlem},
\[
H^0(A^!)^\mathrm{op} \to H^0(A)
\]
induced by taking $H^0$. Clearly it restricts to the subalgebras
\begin{displaymath}
\begin{aligned}
\Hom_{\mathcal{D}(A)}(S_i,S_i) \to \Hom_{H^0(A)}(H^0(S_i),H^0(S_i)) & \simeq \Hom_{H^0(A)}(H^0(A_i),H^0(A_i)) \\ & \simeq D_i
\end{aligned}
\end{displaymath}
We note that this map reflects units. Indeed if $f: S_i \to S_i$ is such that $H^0(f)$ is an isomorphism, then since the cohomology of $S_i$ only lives in degree zero, $f$ is a quasi-isomorphism. Since $D_i$ is a division algebra this implies that the kernel of the above map is exactly the set of non-units. Therefore the non-units form an ideal and $\Hom_{\mathcal{D}(A)}(S_i,S_i)$ is local. 
\end{proof}

\begin{prop} \label{simplesimplebij}
Let $A$ be a coconnective dg algebra with $H^0(A)$ a product of division algebras. Then there is an isomorphism $H^0(A^!)/\rad H^0(A^!) \simeq H^0(A)^\mathrm{op}$. Therefore the simple $H^0(A^!)$ modules are exactly the indecomposable summands of $H^0(\sigma_{\leq 0} A)$.
\end{prop}

\begin{proof}
Let $J \coloneqq \rad H^0(A^!)$. Since $H^0(A^!)$ is semiperfect by \Cref{semiperf}, it follows that $H^0(A^!)/J$ is the maximal semisimple quotient of $H^0(A^!)$. Hence the map $H^0(A^!)^\mathrm{op} \twoheadrightarrow H^0(A)$ from the proof of \Cref{coconnlem} factors as a map $H^0(A^!)^\mathrm{op}/J \twoheadrightarrow H^0(A)$. Since $H^0(A^!)$ is semiperfect, $H^0(A^!)/ J$ is isomorphic to a product of $N$ matrix rings over division algebras. Here $N$ is the number of isomorphism classes of simples, which equals the number of isomorphism classes of indecomposable projectives. As it is semisimple, the only quotient rings of $H^0(A^!)/J$ are products of its connected components, and so $H^0(A)^{\op}$ must be some product of connected components of $H^0(A^!)/J$. However by \Cref{simpleprojbij}, the number of connected components of $H^0(A)^{\op}$ and of $H^0(A)/J$ are equal. It follows that the map $H^0(A^!)^\mathrm{op}/J \twoheadrightarrow H^0(A)$ is an isomorphism. Since $H^0(\sigma_{\leq 0} A) \simeq H^0(A)$ as $H^0(A^!)^{\mathrm{op}}$-modules the second claim also follows.
\end{proof}

The proof of \Cref{semiperf} also shows the following: 

\begin{cor} \label{cor:divimplieslocal}
If $A$ is a coconnective dg algebra with $H^0(A)$ a division algebra, then $H^0(A^!)$ is local.
\end{cor}

If a commutative noetherian ring is complete at a maximal ideal, then it must be local. We conclude a similar result along these lines:

\begin{cor}
Let $B$ be a connective dg algebra augmented over $k$. If $B$ is derived complete at $k$, then $H^0(B)$ is local. 
\end{cor}

\begin{proof}
If $B$ is connective then $B^! \coloneqq \R\hom_B(k,k)^\mathrm{op}$ is coconnective as follows from the standard $t$-structure. Furthermore 
\[
H^0(B^!) = \Hom_{\mathcal{D}(B)}(k,k)^\mathrm{op} \simeq \Hom_{H^0(B)}(k,k)^\mathrm{op} \simeq k
\]
 and so by \Cref{cor:divimplieslocal} we see that $H^0(B) \simeq H^0(B^{!!})$ is local. 
\end{proof}

\begin{rmk} The results of this section should be compared to the extensive literature relating simple-minded collections, silting objects, $t$-structures, and weight structures. See for example \cite{Al-Nofayee, nicolaskeller, KoenigYang, bonfert, fushimi}. A simple-minded collection in an algebraic triangulated category is the same information as a coconnective dg algebra with $H^0$ a product of division algebras, and a silting object is the same as a connective dg algebra.
\end{rmk}

\subsection{Reflexivity of coconnective dg algebras with semisimple $H^0$}

\begin{thm}\label{lem: coconnective semisimple H^0 = reflexive}
Let $A$ be a coconnective dg algebra with $H^0(A)$ a finite dimensional semisimple $k$-algebra. Then $A$ is derived complete if and only if it is reflexive. 
\end{thm}

\begin{proof}\Cref{divalg} ensures the existence of a coconnective dg algebra $A'$, Morita equivalent to $A$, such that $H^0(A')$ is a product of division algebras. In particular $A$ is reflexive if and only if $A'$ is. Moreover, across the induced equivalence $\pvd(A)\simeq \pvd(A')$, the proof of \textit{op.~cit.}\ shows that $\sigma_{\leq 0}A$ corresponds to $\sigma_{\leq 0}A'$, and hence $A$ is derived complete if and only if $A'$ is. So replacing $A$ by $A'$ we may assume that $H^0(A)$ is a product of division algebras. Set $S: = \sigma_{\leq 0} A$. By \Cref{prop: keller-nicolas locality}, $S$ is a $\pvd$-generator for $A$. So by \Cref{lem: two out of three}, it is enough to show that $S$ is a $\pvd$-generator for $A^!$. By \cite[Lemma 6.2]{nicolaskeller}, $A^!$ is connective so by \Cref{prop: Dfd generator connective case} the simple $H^0(A^!)$-modules generate $\pvd(A^!)$. By \Cref{simplesimplebij}, the simple $H^0(A^!)$-modules are summands of $H^0(S)$. Therefore $S \simeq H^0(S) \in \mathcal{D}(A^!)$ is a $\pvd$-generator for $A^!$.
\end{proof}

\begin{cor}\label{propcoconncor}
If $A$ is a proper coconnective dg algebra with $H^0(A)$ semisimple, then $A$ is reflexive.
\end{cor}
\begin{proof}
Follows from \Cref{proper is derived complete} and \Cref{lem: coconnective semisimple H^0 = reflexive}.
\end{proof}

\begin{defn}
    Say that a dg algebra $A$ is \textbf{locally proper} if each $H^i(A)$ is a finite dimensional vector space.
\end{defn}

\begin{prop}\label{FushimiCoConn}
    Let $A$ be a locally proper coconnective dg algebra with $H^0(A)$ semisimple. If $H^0(A^!)$ is finite dimensional, then $A$ is reflexive.
\end{prop}
\begin{proof}
    $A$ is derived complete by \cite[Theorem 4.17]{Fus24}, and hence $A$ is reflexive by an application of \Cref{lem: coconnective semisimple H^0 = reflexive}.
\end{proof}
\begin{rmk}
    In \cite{Fus24}, a locally proper dg algebra is called \textbf{positive} if it is connective and has semisimple $H^0$.
\end{rmk}


\begin{cor}\label{smoothcoconncor}
Let $A$ be a smooth coconnective locally proper dg algebra such that $H^0(A)$ is semisimple. Then $A$ is reflexive.
\end{cor}

\begin{proof}
By \Cref{FushimiCoConn}, it suffices to show that $H^0(A^!)\cong \Hom_{\mathcal{D}(A)}(S,S)$ is a finite dimensional algebra. Since $A$ is smooth, we have an inclusion $\pvd(A) \subseteq \per(A)$ by e.g.\ \cite[Lemma 3.8]{KS}. But since $A$ is locally proper, the triangulated category $\mathcal{D}^{\perf}(A)$ is hom-finite, i.e.\ the space of morphisms between any two objects is finite dimensional. 
\end{proof}

\begin{rmk}\label{finiteglobdimcoconncor}
When $k$ is not perfect, smoothness can be ill-behaved and a more natural replacement is \textbf{hfd-closure}: namely, a dg algebra $A$ is hfd-closed when $\pvd(A) \subseteq \per(A)$. Clearly the proof of \Cref{smoothcoconncor} still applies in this greater generality.
\end{rmk}

\begin{rmk}\label{MYHrmk}
    In the situation of \Cref{smoothcoconncor}, we obtain an isomorphism of derived Picard groups $\mathrm{DPic}(A)\cong \mathrm{DPic}(A^!)$. This generalises the main result of \cite{MaoYangHe}, which assumes that $H^0(A)\cong k$, and in addition that $A^!$ is concentrated in degree zero.
\end{rmk}

\begin{cor}\label{thm: reflexivity via KD for coconn. algs}
 Let $A$ be a locally proper coconnective dg algebra with $H^0(A)$ semisimple. If $H^1(A) \cong 0$ then $A$ is reflexive.
\end{cor}

\begin{proof}
Again, by \Cref{FushimiCoConn} it suffices to show that $H^0(A^!)$ is finite dimensional. The truncation triangles of the co-t-structure for $A$ give an exact sequence
\[
\Hom_{\mathcal{D}(A)}( \sigma_{>0} A,\Sigma^{-1} S) \to \Hom_{\mathcal{D}(A)}(S,S) \to  \Hom_{\mathcal{D}(A)}(A,S) \to \Hom_{\mathcal{D}(A)}(\sigma_{>0} A,S)
\]
Since $\sigma_{>0} A \in \mathcal{D}(A)^{w > 0}$ and $S \in \mathcal{D}(A)^{w \leq 0}$ the last term vanishes. As $H^1(A) \cong 0$, we have $H^1(\sigma_{>0} A) \cong 0$ and so $\sigma_{>0} A \in \mathcal{D}(A)^{w > 1}$. Since $\Sigma^{-1} S \in \mathcal{D}(A)^{w \leq 1}$ the first term vanishes. Therefore $H^0(A^!) \simeq H^0(S) \simeq H^0(A)$ is finite dimensional, as required.
\end{proof}

\begin{ex}\label{ex: SMCs in general} Let $\mathcal{T}$ be a pretriangulated dg category whose underlying triangulated category is hom-finite and admits a simple-minded collection in the sense of \cite{KoenigYang}.

\begin{enumerate}
\item  By \Cref{smoothcoconncor}, if $\mathcal{T}$ is smooth then it is reflexive. 

\item By \Cref{propcoconncor}, if $\mathcal{T}$ is proper, then it is reflexive. For example, this holds if $\mathcal{T}$ is $\thick(S_1,\dots, S_n) \subseteq \mathcal{D}(\mathcal{A})$ where the $S_i$ are simple objects of finite projective or injective dimension in a hom-finite abelian category $\mathcal{A}$.  

\item Recall that a simple-minded collection $\{S_1,\ldots,S_n\}$ is \textbf{rigid} if  $\mathrm{Ext}^1_\mathcal{T}(S_i,S_j)$ vanishes for all $i,j$. By \Cref{thm: reflexivity via KD for coconn. algs}, if $\mathcal{T}$ admits a rigid simple-minded collection, then it is reflexive.
\end{enumerate}
\end{ex}

\begin{ex}
Let $A$ be a finite dimensional algebra which admits a grading with $A/\rad(A)$ in degree zero and $\rad A$ in positive degrees. Then $A$ can be viewed as a formal coconnective dg algebra with semisimple $H^0$, which is hence reflexive. This includes exterior algebras, truncated polynomial rings, and many algebras given by path algebras of quivers with relations.
\end{ex}

\begin{ex}
Let $A$ be a connected graded $k$-algebra ie. $A^i = 0$ for $i < 0$, $A^0 = k$ and $A^i$ is finite dimensional for all $i$. 
\begin{enumerate}
\item If $A^1 = 0$, then by \Cref{thm: reflexivity via KD for coconn. algs}, $A$ is reflexive when viewed as a dg algebra with zero differential. This includes for example the case when $A$ is concentrated in even degrees.

\item If the underlying algebra of $A$ is right Noetherian and has finite global dimension, then when $A$ is viewed as a dg algebra with zero differential it is hfd-closed. Indeed, the above assumptions imply that every finite dimensional graded right $A$-module admits a finite resolution by finitely generated projective right $A$-modules. The claim then follows by lifting these resolutions from cohomology, as in \cite[Section 3]{Kel94}. Hence $A$ is reflexive by \Cref{smoothcoconncor} and \Cref{finiteglobdimcoconncor}.
\end{enumerate}
\end{ex}

\subsection{Reflexivity for connective dg algebras}

\begin{thm}\label{connectiveDC}
Let $A$ be a connective dg algebra such that $H^0(A)$ is finite dimensional. Then $A$ is reflexive if and only if it is derived complete at $H^0(A)/\rad H^0(A)$.
\end{thm}
\begin{proof}For brevity put $R:=H^0(A)/\rad H^0(A)$, so that $R$ is a finite dimensional semisimple $k$-algebra. By \Cref{congen}, $R$ is a $\pvd$-generator for $A$. We wish to apply \Cref{lem: two out of three}, for which it will suffice to show that $R$ is a $\pvd$-generator for $A^!:= \R\hom_A(R,R)^{\op}$. The standard $t$-structure implies that $A^!$ is coconnective, and furthermore we have isomorphisms
\[
H^0(A^!) \cong \Hom_{\mathcal{D}(A)}(R,R)^{\op} \cong \Hom_{H^0(A)}(R,R)^{\op}  = \Hom_R(R,R)^{\op} \cong R^{\op}
\]
which is a semisimple finite dimensional $k$-algebra. Since $\sigma_{\leq 0} A^!$ is a $\pvd$-generator for $A^!$ by \Cref{prop: keller-nicolas locality}, it is enough to show that we have an $A^!$-linear quasi-isomorphism $\sigma_{\leq 0} A^! \simeq R$. But we already have an $H^0(A^!)$-linear isomorphism $H^0(\sigma_{\leq 0} A^!) \cong R$, and by a similar argument to the proof of \Cref{simpleprojbij}, this can be lifted to a quasi-isomorphism of $A^!$-modules $\sigma_{\leq 0}A^! \simeq R$, as required. 
\end{proof}

\begin{rmk}
Note that \Cref{connectiveDC} also holds under the weaker assumption that $H^0(A)$ is semiperfect and $H^0(A)/\rad H^0(A)$ is finite dimensional.
\end{rmk}

\begin{cor} \label{thm: reflexivity via KD for conn. algs}
Let $A$ be a locally proper connective dg algebra. Then $A$ is reflexive.
\end{cor}

\begin{proof}
$A$ is derived complete by \cite[Theorem 4.17]{Fus24}, and hence reflexive by \Cref{connectiveDC}.
\end{proof}

\begin{ex}\label{siltingcor}
If $\mathcal{T}$ is a pretriangulated dg category whose underlying triangulated category is hom-finite and admits a silting object then $\mathcal{T}$ is reflexive. 
\end{ex}

\begin{ex}[Relative singularity categories]\label{ex: rel sing cats}
    Let $R$ be a commutative Gorenstein algebra over an algebraically closed field of characteristic zero. Suppose that $A=\mathrm{End}_R(R\oplus M)$ is a noncommutative resolution of $R$ in the sense of \cite{singcats}, with $M$ a basic MCM $R$-module. Let $e=\mathrm{id}_R \in A$ be the obvious idempotent, so that $A/AeA$ is the stable endomorphism algebra $\underline{\mathrm{End}}_R(M)$. Since $M$ was basic, $A/AeA$ is a nilpotent extension of a finite dimensional $k$-algebra of the form $k\times\cdots \times k$. Then \cite{singcats} gives a Morita equivalence between the derived quotient $A/^\mathbb{L}AeA$ and the relative singularity category $\Delta_R(A)\coloneqq \per(A) / \thick(eA)$. When $R$ is a complete local isolated hypersurface singularity, then the singularity category of $R$ is hom-finite and hence $A/^\mathbb{L}AeA$ is locally proper by \cite[Theorem 6.4.6]{singcats} (cf.\ the proof of \cite[Proposition 8.1.4]{singcats}). In particular, $A/^\mathbb{L}AeA$ is reflexive by \Cref{thm: reflexivity via KD for conn. algs} and hence the dg category $ \Delta_R(A)$ is also reflexive.
\end{ex}

\begin{rmk}
    \Cref{thm: reflexivity via KD for conn. algs} can be interpreted as the Koszul dual of \Cref{FushimiCoConn}, as per \cite[Theorem 4.3]{Fus24}.
\end{rmk}

\section{Commutative rings}

In this section we show that a commutative noetherian $k$-algebra is reflexive if and only if it is a finite product of complete local $k$-algebras, each with residue field finite over $k$. If $R$ is a commutative $k$-algebra and $\mathfrak{m}$ a maximal ideal, we let $k(\mathfrak{m}) \coloneqq R/\mathfrak{m}$ denote the residue field of $R$ at $\mathfrak{m}$. We let $\mathrm{Fin}(R) \subseteq \spec(R)$ denote the set of those maximal ideals $\mathfrak{m} \subseteq R$ such that $k(\mathfrak{m})$ is a finite extension of $k$.
\begin{ex}\label{fintypeex}
    When $R$ is finite type over $k$, then $\mathrm{Fin}(R)$ is simply $\mathrm{MaxSpec}(R)$, the set of maximal ideals of $R$.
\end{ex}

\begin{lem} \label{lemmafdsplit}
Let $R$ be a commutative $k$-algebra. There is a quasi-equivalence 
\[
\pvd(R) \simeq \bigoplus_{\mathfrak{m}\in \mathrm{Fin}(R)}  \thick_R(k(\mathfrak{m})) 
\]
where the right hand side denotes the orthogonal sum of triangulated categories.
\end{lem}
\begin{proof}
We first claim that $\pvd(R)= \thick_R\{k(\mathfrak{m})\colon \mathfrak{m} \in \mathrm{Fin}(R)\}$.
If $k(\mathfrak{m})$ is a finite extension of $k$, then it is certainly contained in $\pvd(R)$, and so the right hand side is contained in the left. On the other hand, if $M \in \pvd(R)$, then $M \in \thick_R( H^\ast(M))$ by an induction argument using the standard $t$-structure, and so it is enough to show that every finite dimensional $R$-module is contained in the right hand side. Every finite dimensional $R$-module $M$ is finite length, and so admits a finite filtration 
\[
0 = M_0 \subseteq M_1 \subseteq \dots \subseteq M_n = M
\]
with each $M_i/M_{i-1} \simeq k(\mathfrak{m}_i)$ for some maximal ideals $\mathfrak{m}_i$. Since $M$ is finite dimensional over $k$, each of the submodules $M_i$ are finite dimensional over $k$, and hence each $\mathfrak{m}_i$ is contained in $\mathrm{Fin}(R)$. Hence $M$ is contained in $\thick_R\{k(\mathfrak{m})\colon \mathfrak{m} \in \mathrm{Fin}(R)\}$ as desired. Finally we just need to show that $\thick_R\{k(\mathfrak{m})\colon \mathfrak{m} \in \mathrm{Fin}(R)\}$ splits as the orthogonal sum $\bigoplus_{\mathfrak{m}\in \mathrm{Fin}(R)}  \thick_R(k(\mathfrak{m}))$. This follows since for all distinct $\mathfrak{m},\mathfrak{m}' \in \mathrm{Fin}(R)$, we have  $\Ext_R^i(k(\mathfrak{m}),k(\mathfrak{m}'))\cong 0$ for all $i$ since $k(\mathfrak{m})$ and $k(\mathfrak{m'})$ have disjoint support.
\end{proof}

\begin{lem} \label{lemmareducetolocal}
Let $R$ be a non-zero reflexive connected commutative $k$-algebra. Then exactly one of the residue fields of $R$ is a finite extension of $k$.
\end{lem}
\begin{proof}
Using the splitting of \Cref{lemmafdsplit} we obtain a quasi-equivalence 
\[
\pvd(R) \simeq \bigoplus_{\mathfrak{m}\in \mathrm{Fin}(R)} \mathcal{D}^{\perf}(\R\mathrm{End}_R(k(\mathfrak{m}))).
\]
If $\mathrm{Fin}(R)$ has more than two elements, then $\pvd(R)^\mathrm{op} \simeq \pvd(R)$ admits a non-trivial orthogonal decomposition, cf.~\Cref{def: semiorthogonal decomposition}. As in \cite[Lemma 3.7]{KS} it follows that $\pvd(\pvd(R)^\mathrm{op})$ admits an orthogonal decomposition. It is non-trivial since each $\pvd(\R\mathrm{End}_R(k(\mathfrak{m})))$ is nonzero; for example it contains $k(\mathfrak{m})$. Since $R$ was reflexive by assumption, it follows that $\mathcal{D}^{\perf}(R) \simeq \pvd(\pvd(R)^\op)$ admits a non-trivial orthogonal decomposition. It follows that $R$ is a decomposable $R$-module, and hence $R$ must be disconnected, which is a contradiction. If none of the residue fields of $R$ are finite extensions of $k$, then $\pvd(R)  \simeq 0$ and $R$ cannot be reflexive.
\end{proof}

\begin{thm} \label{noeththm} Let $R$ be a commutative noetherian $k$-algebra. Then $R$ is reflexive if and only if it is a finite product of complete local $k$-algebras whose residue fields are finite extensions of $k$.
\end{thm}
\begin{proof}
Since $R$ is noetherian, it is a finite product of connected $k$-algebras, and $R$ is reflexive if and only if each of its connected components is (since $\per (R)$ splits as the product of the perfect derived categories of the connected components of $R$). So without loss of generality we may assume that $R$ is connected and non-zero. 

If $R$ is reflexive, then \Cref{lemmareducetolocal} implies that $R$ possesses a unique maximal ideal $\mathfrak{m}$ such that its residue field $K\coloneqq R/\mathfrak{m}$ is a finite extension of $k$. By \Cref{lemmafdsplit}, $K$ is a $\pvd$-generator for $R$. Note that $R_K^! = \R \hom_R(K,K)^\mathrm{op}$ is a coconnective dg algebra with $H^0(R^!_K) = K$, which can be computed using the minimal projective resolution for $K$. As a consequence, we may also assume that $(R_K^!)^0 = K$. By \Cref{prop: keller-nicolas locality} it follows that $\sigma_{\leq 0}R_K^! = (R_K^!)^0$ is a $\pvd$-generator for $R^!_K$.  Hence $K$ is a $\pvd$-generator for $R^!_K$. Since $R$ is reflexive, it now follows from \Cref{lem: two out of three} that the natural map $R \to R_K^{!!}$ is an isomorphism. By \Cref{dercomprmk}(1), this is the case if and only if $R$ is $\mathfrak{m}$-adically complete. In particular, $R$ must be local and $\mathfrak{m}$ its unique maximal ideal. The converse is similar. 
\end{proof}

\begin{rmk}
Let $R$ be a non-regular complete local $k$-algebra of Krull dimension at least one at residue field $k$. Then $R$ is reflexive by \Cref{noeththm}, but neither $R$ nor $R^!_K$ are proper. This gives to our knowledge the first example of a reflexive dg category such that neither $R$ nor $\pvd(R)$ are proper.
\end{rmk}

\begin{ex}
Suppose that $R$ is finite type over $k$ and reflexive. Then $R$ is a finite product of Artinian local $k$-algebras: by \Cref{fintypeex} and \Cref{noeththm}, it is a finite product of complete local $k$-algebras. But a complete $k$-algebra is finite type exactly when it is Artinian. 
\end{ex}

\begin{rmk} The field extension condition cannot be dropped: if $K$ is an infinite field extension of $k$ then certainly $K$ is complete local, but not a reflexive $k$-algebra. However, via the Cohen structure theorem, the assumption on the residue field $K$ can be dropped in one direction of \Cref{noeththm} by changing the base field. Indeed, let $R$ be a commutative complete local noetherian $k$-algebra with residue field $K$. Then the Cohen structure theorem tells us that $R$ is a complete local $K$-algebra, and hence by \Cref{noeththm} $R$ is a reflexive $K$-algebra.
\end{rmk}

\begin{rmk}
    Let $(R,\mathfrak{m},K)$ be a local noetherian $k$-algebra. If $\mathrm{Kos}(\mathfrak{m})$ denotes the Koszul complex of $\mathfrak{m}$, then it follows from the main result of {\cite{PSY}} that $R$ is complete local if and only if it is derived complete with respect to $\mathrm{Kos}(\mathfrak{m})$. In particular, if $K$ is finite over $k$, then $R$ is reflexive if and only if it is $\mathrm{Kos}(\mathfrak{m})$-derived complete. We remark that in this setting, $\mathrm{Kos}(\mathfrak{m})$ is a $\pvd$-generator for $R$ precisely when $R$ is regular: in one direction, simply use that $\mathrm{Kos}(\mathfrak{m}) \simeq K$ when $R$ is regular. In the other direction, note that $\mathrm{Kos}(\mathfrak{m})$ is perfect, and hence if it is a $\pvd$-generator then $K$ is perfect, which implies that $R$ is regular.
\end{rmk}

\begin{rmk}\label{EfimovRemark}
   We note a comparison to the derived completion along $\pvd$ in the sense of \cite{efimovcompletion}. If $R$ is a commutative noetherian ring, then $\hat{R}_{\pvd(R)}$ is the one-object dg category $\prod_{\mathfrak{m}\in \mathrm{Fin}(R)} \hat{R}_\mathfrak{m}$, whereas $\pvd\pvd(R)$ need not have a single compact generator. Observe that the map $R \to \hat{R}_{\pvd(R)}$ is precisely the map of dg algebras $\mathrm{End}_{\per(R)}(R) \to \mathrm{End}_{\pvd(\pvd(R))}(\mathrm{ev}R)$ induced by the evaluation functor. We note also that $\hat{R}_{\pvd(R)}$ is the pseudocompact completion of $R$.
\end{rmk}

\section{Reflexivity and Koszul duality}\label{kdsection}

In this section we use Koszul duality to understand reflexivity. The key point for us will be that the Koszul double dual of an augmented dg algebra computes its derived completion along $k$, a viewpoint which was first taken in \cite{DGI}. We begin by recalling facts about dg coalgebras and the bar and cobar constructions. We then define what it means for a dg coalgebra to be reflexive, and compare this notion to the one for algebras across both linear and Koszul duality.

\subsection{Coalgebras, comodules, coderived categories}
We begin by quickly recalling some standard facts about dg coalgebras and their coderived categories. For comprehensive references on this section and the next, see e.g.\ \cite{lefevre, positselski, lodayvallette}. Throughout this section, we let $R$ be a commutative finite dimensional semisimple $k$-algebra (i.e.\ a finite product of finite field extensions of $k$). If $M$ is a dg $R$-module then we denote by $M^\vee\coloneqq \hom_R(M,R)$ its $R$-linear dual.

A \textbf{dg-$R$-coalgebra} is a comonoid in the symmetric monoidal category of dg-$R$-modules: explicitly it is a complex of $R$-modules $C$ with a coassociative comultiplication $\Delta\colon C\to C\otimes_R C$ and a counit $\eta\colon C\to R$. The condition that $\Delta$ is a chain map translates into the condition that $d$ is a coderivation for $\Delta$. A dg coalgebra is \textbf{coaugmented} if $\eta$ admits a section which is a morphism of $R$-coalgebras. In this case, the coaugmentation coideal $\bar C\coloneqq \ker \eta$ becomes a noncounital dg coalgebra under the reduced comultiplication $\bar\Delta$. Say that a coaugmented dg coalgebra $C$ is \textbf{conilpotent} if for all $c\in \bar C$ there exists $N\in \mathbb{N}$ such that $\bar{\Delta}^N(c)=0$.

If $C$ is a dg-$R$-coalgebra then its $R$-linear dual $C^\vee$ is a dg-$R$-algebra under the operations $\eta^\vee$ and $\Delta^\vee$. In fact, $C^\vee$ is a \textbf{pseudocompact} dg-$R$-algebra, meaning a topological dg algebra obtained as an inverse limit of discrete finite dimensional dg algebras, equipped with the inverse limit topology. If $C$ is conilpotent then $C^\vee$ is \textbf{pronilpotent}, meaning that the finite dimensional algebras occurring in the inverse limit are all nilpotent extensions of $R$. The $R$-linear dual functor gives an equivalence of categories between dg-$R$-coalgebras and pseudocompact dg-$R$-algebras, and between conilpotent dg-$R$-coalgebras and pronilpotent dg-$R$-algebras.

If $C$ is a dg-$R$-coalgebra, a (right) \textbf{dg-$C$-comodule} is a dg-$R$-module $V$ together with a coaction map $\rho\colon V\to V\otimes C$ such that $(\mathrm{id}_V\otimes \Delta)\rho = (\rho\otimes \mathrm{id}_C)\rho$. A \textbf{$C$-colinear map} between two dg-$C$-comodules $U,V$ is an $R$-linear map $U\to V$ which is compatible with the coactions. These define an abelian dg category $C\mathbf{-Comod}$ of dg-$C$-comodules. We let $\mathrm{Hot}(C)\coloneqq H^0(C\mathbf{-Comod})$ denote the corresponding homotopy category; it is a triangulated category. The subcategory $\mathbf{CoAcy}(C)$ of \textbf{coacyclic} $C$-comodules is then defined to be the smallest localising subcategory of $\mathrm{Hot}(C)$ containing the totalisations of exact triples of $C$-comodules. The \textbf{coderived category} of $C$ is the Verdier quotient $\dco(C)\coloneqq \mathrm{Hot}(C)/\mathbf{CoAcy}(C)$. By \cite[\S5.5]{positselski} the triangulated category $\dco(C)$ is compactly generated by the full subcategory $\fd(C) \into \dco(C)$ on those comodules weakly equivalent to finite dimensional comodules. When $C$ is conilpotent over $R$ there is a natural equivalence $\fd(C)\simeq \thick_{\dco(C)}(R)$.

One can also recover $\dco(C)$ as the homotopy category of a model structure on the category of $C$-comodules; the cofibrations are the injections and the weak equivalences are the maps with coacyclic cone. In particular every comodule is cofibrant. Via taking dg quotients one can also enhance $\dco(C)$ to a pretriangulated dg category, and we will frequently regard it as such.

A weak equivalence between comodules is a quasi-isomorphism, but the converse is not true. In particular, there is a quotient map $\dco(C) \to {\mathcal{D}}(C\mathbf{-Comod})$ which is full and essentially surjective, but not faithful (fullness follows from Brown representability). Recalling that the $R$-linear dual $C^\vee$ is a pseudocompact dg algebra, the linear dual functor gives a contravariant equivalence between $C\mathbf{-Comod}$ and the category $C^\vee\mathbf{-pcMod}$ of pseudocompact $C^\vee$-modules. Thus, we  obtain a natural quasi-equivalence of pretriangulated dg categories between ${\mathcal{D}}(C\mathbf{-Comod})^\op $ and $ {\mathcal{D}}(C^\vee\mathbf{-pcMod})$. The forgetful functor from pseudocompact $C^\vee$-modules to all $C^\vee$-modules induces a functor ${\mathcal{D}}(C^\vee\mathbf{-pcMod}) \to {\mathcal{D}}(C^\vee)$ which is neither essentially surjective nor full. Hence by combining the above functors we obtain an $R$-linear dual functor $\dco(C)^\op \to {\mathcal{D}}(C^\vee)$ which sends a comodule $N$ to its $R$-linear dual $N^\vee$. In general this functor need not be full, faithful, or essentially surjective.

\subsection{Bar and cobar constructions}

 Let $A$ be a dg-$R$-algebra. Recall that $A$ is \textbf{augmented} if there is an $R$-algebra map $A \to R$ splitting the unit. In this case, the augmentation ideal $\bar A \coloneqq \ker(A\to R)$ becomes a nonunital dg algebra.

If $V$ is a dg-$R$-module, its \textbf{tensor coalgebra} is the dg-$R$-coalgebra given by $T_R^c(V)\coloneqq R\oplus V \oplus \left(V\otimes_R V\right)\oplus \cdots$ with comultiplication given by the deconcatenation coproduct, which sends a tensor $v_1\otimes\cdots\otimes v_n$ to the sum $\sum_i(v_1 \otimes \cdots \otimes v_i)\otimes(v_{i+1}\otimes\cdots\otimes v_n)$. It is not hard to see that $T^c_RV$ is a conilpotent dg-$R$-coalgebra. In fact, $T^c_R(V)$ is the cofree conilpotent coalgebra on $V$, in the sense that $T^c_R$ is the right adjoint to the forgetful functor from conilpotent dg-$R$-coalgebras to dg-$R$-modules.

If $A$ is an augmented dg-$R$-algebra, then its \textbf{bar construction} is the conilpotent dg coalgebra $B_RA$ whose underlying graded coalgebra is $T^c_R(\bar{A}[1])$. The differential combines the usual differential on the tensor coalgebra with the multiplication on $A$. We denote the $R$-linear dual of $B_RA$ by $B^\vee_RA$; it is a pseudocompact dg-$R$-algebra.

\begin{ex}
    If $A$ is the square-zero extension $R[\epsilon]/\epsilon^2$, with $\epsilon$ in degree $1-n$, then $B^\vee_RA$ is the pseudocompact algebra $R\llbracket x \rrbracket$ with $x$ placed in degree $n$. More generally, the square-zero extension $R\oplus V$ has dual bar construction given by the free pseudocompact dg-$R$-algebra on the pseudocompact dg-$R$-module $V^\vee[-1]$.
\end{ex}

There is a dual construction, the {cobar construction}, which sends conilpotent coalgebras to algebras. In brief, let $C$ be a conilpotent dg-$R$-coalgebra. The \textbf{cobar construction} of $C$ is the dg algebra $\Omega C$ whose underlying graded algebra is $T(\bar C[-1])$, the tensor algebra on the shifted augmentation ideal of $C$. The differential combines the natural internal differential on the tensor algebra with the comultiplication on $C$.

\begin{thm}[Koszul duality, cf.\ {\cite{positselski}}]\hfill
\begin{enumerate}
\item $\Omega$ and $B$ form an adjunction $\Omega\colon \mathbf{dgCog}_R^\mathrm{conil} \longleftrightarrow \mathbf{dgAlg}_R^\mathrm{aug}\colon B$.
\item Let $A$ be an augmented dg-$R$-algebra. The counit $\Omega BA \to A$ is a quasi-isomorphism of algebras.
\item Let $C$ be a conilpotent dg-$R$-coalgebra. The unit $C \to B\Omega C$ is a weak equivalence of coalgebras (i.e.\ is sent to a quasi-isomorphism by $\Omega$).
\end{enumerate}
\end{thm}

\begin{rmk}
    In fact, Positselski proves that there exist model structures on the categories of augmented dg algebras and conilpotent dg coalgebras making $\Omega \dashv B$ into a Quillen equivalence. A weak equivalence of algebras is precisely a quasi-isomorphism. The weak equivalences of coalgebras are created by $\Omega$; every weak equivalence is a quasi-isomorphism but the converse is not true.
\end{rmk}

Following \cite{BCLcalabiyau} we will refer to any pair $(C,A)$ consisting of a conilpotent dg-$R$-coalgebra $C$ and an augmented dg-$R$-algebra $A$ such that $\Omega C \simeq A$ as a \textbf{Koszul duality pair}. The prototypical examples of Koszul duality pairs are pairs of the form $(C,\Omega C)$ and $(BA,A)$, and up to pairwise weak equivalence every pair is of this form.

\begin{thm}[Module-comodule Koszul duality, cf.\ {\cite{positselski}}]
    If $(C,A)$ is a Koszul duality pair then there is a quasi-equivalence of pretriangulated dg categories $\dco(C)\simeq {\mathcal{D}}(A)$ which sends $R$ to $A$ and $C$ to $R$.
\end{thm}

Module-comodule Koszul duality can be used to show that the dual bar construction computes derived endomorphisms. Let $C$ be a conilpotent dg coalgebra and $M,N$ two $C$-comodules. We write $\R\hom_C(M,N) \in {\mathcal{D}}(k)$ for the derived mapping space between $M$ and $N$ computed in the dg category $\dco(C)$. This can be computed as the complex of $C$-colinear morphisms $\hom_C(M,\tilde N)$, where $\tilde N$ is a fibrant replacement of $N$ (note that since every comodule is cofibrant, we do not need to replace $M$). 
\begin{prop}
    Let $A$ be an augmented dg-$R$-algebra. Then there is a dg-$R$-algebra quasi-isomorphism $B^\vee_RA \simeq \R\mathrm{End}_A(R)$.
\end{prop}
\begin{proof}
    Put $C\coloneqq B_RA$. By module-comodule Koszul duality, there are quasi-isomorphisms of dg-$R$-algebras
    $$\R\mathrm{End}_A(R) \simeq \R\mathrm{End}_{{\mathcal{D}}(A)}(R) \simeq \R\mathrm{End}_{\dco(C)}(C)$$and since $C$ is an injective $C$-comodule, it is fibrant, and we can compute its derived endomorphisms as $\R\mathrm{End}_{\dco(C)}(C) \simeq \mathrm{End}_C(C) \cong C^\vee$, as desired.
\end{proof}

In particular, there is a dg-$R$-algebra quasi-isomorphism $B^\vee_RA \simeq (A^!_R)^\op$, which we will implicitly use going forward.

\begin{rmk}
    One can use the bar-cobar formalism to prove versions of \Cref{smoothcoconncor}, \Cref{thm: reflexivity via KD for coconn. algs}, and \Cref{thm: reflexivity via KD for conn. algs} under the extra hypothesis that the dg algebras in question are augmented over $R$. The relevant derived completeness results can be deduced in a similar manner as those of \cite{MBDDP}.
\end{rmk}

\subsection{Reflexivity for coalgebras}
Recall that we write $\R\hom_C(-,-)$ for the derived mapping space functor of $\dco(C)$. The following definition appears in \cite{BCLcalabiyau} and is a homotopical version of Takeuchi's definition of a quasi-finite comodule \cite{takeuchi}.

\begin{defn}
    Let $C$ be a dg coalgebra. Say that a $C$-comodule $M$ is \textbf{homotopy quasi-finite} if, for all $X \in \fd(C)$, the complex $\R\hom_C(X,M)$ is proper. We abbreviate homotopy quasi-finite by \textbf{hqf} and denote the full subcategory of $\dco(C)$ on the hqf comodules by $\hqf(C)$.
\end{defn}

\begin{rmk}
    The category $\hqf$ is Morita invariant, in the sense that if there is a quasi-equivalence $\dco(C) \simeq \dco(C')$ then $\hqf(C)\simeq \hqf(C')$. This is because $\fd(C)$ has a Morita invariant description as the compact objects in $\dco(C)$.
\end{rmk}

\begin{prop}
    Let $(C,A)$ be a Koszul duality pair. Module-comodule Koszul duality induces the following quasi-equivalences of pretriangulated dg categories:
    \begin{enumerate}
        \item $\per(A) \simeq \fd(C)$.
        \item $\pvd(A) \simeq \hqf(C)$.
    \end{enumerate}
\end{prop}
\begin{proof}
This appears in \cite{BCLcalabiyau} but we give a self-contained proof. The quasi-equivalence of (1) is well known and follows from restricting the Koszul duality equivalence ${\mathcal{D}}(A) \simeq \dco(C)$ to compact objects. To show that (2) holds, let $M$ be an $A$-module and let $N$ be its corresponding $C$-comodule. Then we have 
    \begin{align*}
        M\in \pvd(A) &\iff \R\hom_A(\per(A),M)\in \pvd(k)&\\
    &\iff \R\hom_C(\fd(C),N)\in \pvd(k) & \text{by (1)}\\
        &\iff N\in \hqf(C) & \text{by definition. \qedhere}
    \end{align*} 
\end{proof}
Note that there is a natural functor $\mathrm{ev}_C\colon\fd (C)^\op\to\pvd(\hqf (C)^\op)$ that sends a finite dimensional $C$-comodule $X$ to the $\hqf (C)$-module $N \mapsto \R\hom_C(X,N)$. Say that a dg coalgebra $C$ is \textbf{reflexive} if $\ev_C$ is a quasi-equivalence. 
\begin{prop}\label{mrefcog}
    Let $(C,A)$ be a Koszul duality pair. Then $A$ is reflexive if and only if $C$ is reflexive.
\end{prop}
\begin{proof}
    The diagram
$$\begin{tikzcd}
    \per(A)^\op\ar[d,"\ev_A"] \ar[r,"\simeq"]& \fd(C)^\op \ar[d,"\ev_C"]\\
    \pvd (\pvd (A)^\op) \ar[r,"\simeq"]& \pvd (\hqf (C)^\op)
\end{tikzcd}$$commutes, which implies that $\ev_A$ is a quasi-equivalence precisely when $\ev_C$ is. 
\end{proof}

\subsection{Reflexivity and linear duality}
Let $R$ be a finite dimensional commutative semisimple $k$-algebra and let $C$ be a dg-$R$-coalgebra. We wish to study the relationship between reflexivity of $C$ and reflexivity of its $R$-linear dual $C^\vee$. 


Let $A$ be a dg algebra and $M$ a proper $A$-module. Say that $M$ \textbf{admits a finite dimensional model} if there is a finite dimensional $A$-module $M'$ and an $A$-linear quasi-isomorphism $M\simeq M'$.

\begin{prop}\label{linduallem}
Let $(C,A)$ be a Koszul duality pair. Then:
     \begin{enumerate}
         \item Sending $R \mapsto C^\vee$ induces a quasi-equivalence $\thick_A(R) \xrightarrow{\simeq} \per(C^\vee)^\op$.
         \item Across the quasi-equivalence of (1), the restricted evaluation map $\ev_{A,R}$ corresponds to the linear duality map $\fd(C)^\op \to \pvd(C^\vee)$.
         \item The following are equivalent:
         \begin{enumerate}
             \item $A$ is $R$-restricted reflexive.
             \item The linear duality map $\fd(C)^\op \to \pvd(C^\vee)$ is a quasi-equivalence.
         \end{enumerate}
        \item The linear duality map $\fd(C)^\op \to \pvd(C^\vee)$ is quasi-essentially surjective precisely when every proper $C^\vee$-module admits a finite dimensional model.
     \end{enumerate}
\end{prop}
\begin{proof}
    Claim (1) is clear since $C^\vee \simeq \R\mathrm{End}_A(R)$. To prove (2), for brevity put $\mathcal{R} \coloneqq \thick_{{\mathcal{D}}(A)}(R)^\op$. Applying $\pvd$ to the equivalence of (1) hence yields an equivalence $\pvd (\mathcal{R}) \simeq \pvd(C^\vee)$. To check that the diagram $$\begin{tikzcd}
        \per (A)^\op\ar[d,"\ev_{A,R}"] \ar[r,"\simeq"]& \fd(C)^\op \ar[d,"(-)^\vee"] \\
        \pvd (\mathcal{R}) \ar[r,"\simeq"] & \pvd(\per (C^\vee))
    \end{tikzcd}$$
commutes, it is enough to check it on the generator $A$ of $\per(A)$ together with its endomorphisms. But both compositions send the object $A$ to the $\per (C^\vee)$-module which sends $C^\vee$ to $R$, and an element $a\in A$ to its action on $R$. The equivalence of (3a) and (3b) is immediate from the proof of (2). Finally, (4) follows from the fact that the essential image of $\fd(C)^\op \to \pvd(C^\vee)$ consists exactly of those modules with finite dimensional models.
\end{proof}

If $C$ is a conilpotent dg-$R$-coalgebra, observe that there is a natural completion map $\Omega C \to B_R(C^\vee)^\vee$. Following \cite[Proposition 4.1.7]{MBDDP}, this completion map is a quasi-isomorphism when $C$ is finite dimensional in each degree and either connective or 2-coconnective (i.e.\ connected with $C^0\cong k$ and $C^1\cong 0$).

\begin{cor}\label{LDreflex}
Let $(C,A)$ be a Koszul duality pair such that $R$ is a $\pvd$-generator for $A$. If $C$ is reflexive then $C^\vee$ is reflexive. The converse is true as long as $R$ is a $\pvd$-generator for $C^\vee$ and the completion map $A \to B_R(C^\vee)^\vee$ is a quasi-isomorphism.
\end{cor}
\begin{proof}
    If $C$ is reflexive then so is $\Omega C$. Hence $\pvd(\Omega C)$ is also reflexive. But by \Cref{linduallem}(1), $\pvd(\Omega C)$ is Morita equivalent (up to an opposite) to $C^\vee$, and hence $C^\vee$ is reflexive. For the converse, if $C^\vee$ is reflexive then by \Cref{linduallem}(1) again, $\pvd(\Omega C)$ is reflexive. Since $R$ is a $\pvd$-generator for $C^\vee$, there is a Morita equivalence between $(C^\vee)_R^!$ and $\pvd(\Omega C)$, and hence $(C^\vee)_R^!$ is reflexive. But $(C^\vee)_R^!$ is the opposite of $B_R(C^\vee)^\vee$, which by assumption is $\Omega C$. Hence $\Omega C$ is reflexive, and hence $C$ is reflexive, as desired.
\end{proof}

\begin{rmk}
    In the converse situation of \Cref{LDreflex}, one can show that the linear duality map $\fd(C) \to \pvd(C^\vee)^\op$ is a quasi-equivalence. In particular, every proper $C^\vee$-module admits a finite dimensional model.
\end{rmk}

\section{Ginzburg dg algebras}

We use our results on Koszul duality to show that completed Ginzburg dg algebras \cite{ginzburg} and completed undeformed Calabi--Yau completions \cite{kellercyc} are reflexive. A similar treatment appears in \cite[Appendix]{kellercyc} and \cite{hanliuwang}.

\subsection{Calabi--Yau completions}
Let $Q$ be a finite quiver with vertex set $Q_0$ and arrow set $Q_1$. Write $kQ_0$ for the semisimple algebra on the vertex idempotents. The path algebra of $Q$ is then $kQ \coloneqq T_{kQ_0}(kQ_1)$. We denote composition in the path algebra from left to right, so that $ab$ means `follow arrow $a$ then arrow $b$'. If $M$ is a $kQ_0$-module we write $M^\vee$ for its $kQ_0$-linear dual. In particular, the space of arrows $kQ_1$ is a $kQ_0$-bimodule; its dual is the bimodule of `opposite arrows', often written $\overline{kQ_1}\coloneqq kQ_1^\vee$. If $a:u\to v$ is an arrow we denote its corresponding opposite arrow by $a^\vee \colon v \to u$.

Pairing an arrow with its dual gives natural $kQ_0$-linear pairings $kQ_1\otimes_{kQ_0} kQ_1^\vee \to kQ_0$ and $kQ_1^\vee \otimes_{kQ_0} kQ_1 \to kQ_0$ which we denote by $\langle -,-\rangle$. Concretely, if $a,b$ are arrows, then $\langle a, b^\vee\rangle$ is zero unless $a=b$, in which case it is $\mathrm{tail}(a) =\mathrm{head} (b^\vee)$. The pairing $\langle a^\vee, b\rangle$ behaves similarly.

Fix an integer $n$ and let $R_n$ be the graded $kQ_0$-module
$$R_n\coloneqq kQ_0\oplus kQ_1[-1]\oplus kQ_1^\vee[1-n]\oplus kQ_0[-n].$$
We equip $R_n$ with a multiplication making it into a graded $kQ_0$-algebra. Writing a basis element of $R_n$ in the form $(x,a,b^\vee,y)$, the multiplication is given by $$(x,a,b^\vee,y)(u,r,s^\vee,v) = (xu, xr+au,xs^\vee+b^\vee u,\langle a, s^\vee\rangle + \langle b^\vee,r\rangle +xv+yu)$$It is easy to verify that this multiplication makes $R_n$ into a graded $kQ_0$-algebra.

\begin{rmk}\label{cyccomprmk}
Let $S$ be the graded $kQ_0$-algebra given by the square-zero extension $kQ\oplus kQ_1[-1]$. In other words, $S$ is the quotient of the path algebra $kQ$ by the square of the arrow ideal. The grading is given by placing the arrows in degree one. Let $M$ be the graded $S$-bimodule $S^\vee[-n]$, with action given by the pairing $\langle-,-\rangle$. Then $R_n$ is the trivial extension algebra $S\oplus M$. This is a basic example of a construction known as cyclic completion \cite{segal}.
\end{rmk}

 Let $\Pi_n(Q)$ be Keller's (undeformed) Calabi--Yau completion of $kQ$, in the sense of \cite{kellercyc}, and let $\hat\Pi_n(Q)$ be the completion of $\Pi_n(Q)$ at the arrow ideal. Note that both of these are augmented dg-$kQ_0$-algebras. In what follows, if $A$ is an augmented dg algebra we write $A^!\coloneqq B(A)^\vee$. In addition if $A$ is finite dimensional we write $A^\dagger\coloneqq \Omega(A^\vee)$, following the notation of \cite{hanliuwang}.

\begin{lem}\label{rndualispin}For a finite quiver $Q$, there are dg-$kQ_0$-algebra quasi-isomorphisms\begin{enumerate}
    \item $R_n^\dagger\simeq \Pi_n(Q)$.
    \item $R_n^! \simeq \hat\Pi_n(Q)$.
    \item $ \hat\Pi_n(Q)^!\simeq R_n$.
\end{enumerate}

\end{lem}
\begin{proof}
By the definition of the cobar construction, the underlying graded algebra of $R_n^\dagger$ is freely generated over $kQ_0$ by the arrows $kQ_1$ in degree zero, the dual arrows $kQ_1^\vee$ in degree $2-n$, and loops $z_i$, one at each vertex $i$, placed in degree $1-n$. The differential satisfies $da=0, da^\vee=0$, and $dz_i = \sum_a e_i[a,a^\vee]e_i$ where $e_i$ is the idempotent at vertex $i$. Claim (1) then follows from the description of $\Pi_n(Q)$ as a Ginzburg dg category given in \cite[Theorem 6.3]{kellercyc}. Since $R_n^! \simeq B(R_n)^\vee$, Claim (2) follows from Claim (1) by an application of \cite[Proposition 4.1.7]{MBDDP}. Observe that the dg radical of the graded algebra $R_n$ is the ideal $kQ_1[-1]\oplus kQ_1^\vee[1-n]\oplus kQ_0[-n]$. Hence by \Cref{proper is derived complete} we see that $R_n$ is derived complete along $kQ_0$. Claim (3) now follows from Claim (2) using the quasi-isomorphism $R_n\simeq R_n^{!!}$.
\end{proof}

\begin{rmk}Note that the relations $dz_i = \sum_a e_i[a,a^\vee]e_i$ impose the preprojective relations on the cohomology of $\Pi_n(Q)$. In particular, $\Pi_2(Q)$ is the dg preprojective algebra of $Q$. When $Q$ has no cycles and is not of ADE type, then $\Pi_2(Q)$ is a resolution of the classical preprojective algebra of $Q$ \cite{hermes}. 
\end{rmk}

\begin{prop}\label{cycompref}For $n\geq 2$, both $R_n$ and $\hat\Pi_n(Q)$ are reflexive.
\end{prop}
\begin{proof} Since $n\geq 2$, the algebra $R_n$ is coconnective and we have $H^0R_n\cong kQ_0$, which is a finite dimensional semisimple $k$-algebra. Moreover $R_n$ is derived complete at $kQ_0$ by \Cref{rndualispin}. Hence by \Cref{lem: coconnective semisimple H^0 = reflexive}, we see that both $R_n$ and $R_n^!$ are reflexive, as required.
\end{proof}

\begin{rmk}
    When $Q$ has no cycles, the natural map $\Pi_n(Q) \to \hat\Pi_n(Q)$ is an isomorphism (in particular a quasi-isomorphism!), and hence for $n\geq 2$ we see that $\Pi_n(Q)$ is reflexive. When $Q$ is a tree, $n=2$, and $\mathrm{char}(k) \neq 2$ this was proved in \cite{EtguLekili}, who used it to compute the symplectic cohomology of an associated Liouville manifold obtained by plumbing copies of $T^*S^2$ along $Q$. Similar results in higher dimensions were given in \cite{KarabasLee}.
\end{rmk}

\begin{rmk}
    When the quiver $Q$ itself is graded, then $\Pi_n(Q)$ inherits an extra grading, often known as an \textbf{Adams grading}. In good situations, one can use this to prove that $\Pi_n(Q)$ is formal, or complete at the arrow ideal.
\end{rmk}

\begin{rmk}
    For $n\leq 0$, the algebra $R_n$ is neither connective nor coconnective. However, $R_1$ is coconnective, with $H^0(R_1)\cong R^0_1$ a square-zero extension of $kQ_0$. In particular, one should be able to generalise \Cref{cycompref} to the $n=1$ case; ideally this would follow from a similar analysis along the lines of \Cref{weightsection} where one relaxes the condition that $H^0$ be semisimple to the condition that it be Artinian. It is possible that one can use the compactly generated co-t-structures of \cite{pauksztello} to do this (specifically, the co-t-structure generated by $R_1$ itself together with all of its negative shifts).
\end{rmk}

\subsection{Completed Ginzburg algebras}
In the $n=3$ case we now want to turn on a superpotential $W$, to obtain reflexivity results for deformed $3$-Calabi--Yau completions. These were first studied by Ginzburg \cite{ginzburg} and are hence known as Ginzburg dg algebras. In this section, we restrict $k$ to be a characteristic zero field.

Recall that a \textbf{(completed) superpotential} $W$ on a quiver $Q$ is an element of the completed cocentre $\widehat{kQ}/[\widehat{kQ},\widehat{kQ}]$. In simpler terms, a superpotential is a possibly infinite linear combination of cycles, with only a finite number of cycles of any given length occurring. We call a superpotential \textbf{finite} if it has only finitely many terms. Let $W$ be a superpotential and $a$ an arrow of $Q$. We define the \textbf{cyclic derivative} of $W$ with respect to $a$ to be the sum $\partial_aW\coloneqq \sum_{W=uav} vu \in \widehat{kQ}$.

The \textbf{Ginzburg dg algebra} $\Gamma(Q,W)$ associated to a quiver with finite superpotential $(Q,W)$ is defined as follows. The underlying graded algebra is the same as that of $\Omega(R_3^\vee)$. The differential is defined by $da =0$, $da^\vee = \partial_aW$, and $dz_i = \sum_a e_i[a,a^\vee]e_i$. In other words, it is the same as $\Pi_3(Q)$ but the differentials of the opposite arrows are deformed by the superpotential $W$. We have $H^0(\Gamma(Q,W))\cong kQ/(\partial_aW)_{a\in Q_1}$, the \textbf{Jacobi algebra} $\mathrm{Jac}(W)$ associated to the superpotential. Similarly, one can define a \textbf{completed Ginzburg dg algebra} $\hat\Gamma(Q,W)$ from a quiver with superpotential, and $H^0(\hat\Gamma(Q,W))$ is the \textbf{completed Jacobi algebra} $\widehat{\mathrm{Jac}}(W)$.

\begin{rmk}
    There is a natural completion map $\Gamma(Q,W)\to \hat\Gamma(Q,W)$. When $(Q,W)$ is \textbf{Jacobi-finite}, i.e. ${\mathrm{Jac}}(W)$ is a finite dimensional algebra, then this map is a quasi-isomorphism; one can show this in a similar manner to the proof of \cite[Theorem 4.2.6]{MBDDP}.
\end{rmk}

We will use the notation $W=W^{\geq m}$ to mean that all of the cycles appearing in $W$ have length $\geq m$. This ensures that each term in the cyclic derivatives of $W$ has length at least $m-1$.

\begin{prop}[Van den Bergh]
    Let $Q$ be a quiver and $W=W^{\geq 2}$ a superpotential on $Q$. Then there exists the structure of an $A_\infty$-$kQ_0$-algebra $R^W_3$ on the graded vector space $R_3$ and a quasi-isomorphism $B(R_3^W)^\vee \simeq \hat\Gamma(Q,W)$. If $W=W^{\geq 3}$ then $R_3^W$ is minimal (i.e.\ the differential vanishes). If $W=W^{\geq 4}$ then the underlying graded algebra of $R^W_3$ agrees with the previously defined algebra structure on $R_3$. If $W=0$ then $R_3^W=R_3$.
\end{prop}
\begin{proof}
This is \cite[A.15]{kellercyc}. Abstract existence of the $A_\infty$ structure simply follows from the fact that the Ginzburg algebra is a dg algebra. To actually construct the $A_\infty$ products $m_r$, the idea is that to obtain $R_3^W$ from $R_3$, one need only add terms of the form $m_r(a_1,\ldots,a_r) = \pm b^\vee/r$ corresponding to length $r+1$ cycles in $W$ whose cyclic derivative with respect to $b$ is a cyclic permutation of $a_1\cdots a_r$. Note that here we are using that $k$ has characteristic zero. In particular, if $W$ has no 2-cycles then we do not modify the differential of $R_3$, and if $W$ has no 3-cycles then we do not modify the multiplication.
\end{proof}

\begin{rmk}
Cyclic invariance of $W$ ensures that the above constructed $A_\infty$ structure is actually a cyclic $A_\infty$ structure in the sense of \cite{kontsoib}; the relevant inner product on $R_3^W$ is the one described above. A concrete description of the above constructed $m_r$ for the two-loop one-vertex quiver is given in \cite{brownwemyss} in terms of necklace polynomials. 
\end{rmk}

\begin{rmk}
    If $Q$ is a finite quiver and $\lambda = (\lambda_i)_{i\in Q_0}$ is a set of weights on $Q$, then the \textbf{deformed dg preprojective algebra} $\Pi_2(Q,\lambda)$ is defined similarly to $\Pi_2(Q)$, but where we now modify $dz_i$ by a $\lambda_i e_i$ term \cite{kellercyc, kalckyangii}. The relevant modification of $R_2$ now requires a curvature term, which our methods cannot handle.
\end{rmk}

Note that $R_3^W$ is $A_\infty$-quasi-isomorphic to a dg algebra, namely the dg algebra $\Omega B(R_3^W)$. We will abusively say that $R_3^W$ is reflexive to mean that this latter dg algebra is reflexive.

\begin{prop}\label{ginzref}
    Let $Q$ be a finite quiver and $W=W^{\geq 3}$ a superpotential on $Q$. Then both $R_3^W$ and $\hat\Gamma(Q,W)$ are reflexive.
\end{prop}

\begin{proof}
This is similar to the proof of \Cref{cycompref}. Put $A\coloneqq \Omega B(R_3^W)$. Then $A$ is a coconnective dg algebra with finite dimensional semisimple $H^0$. Moreover $A$ is a proper dg algebra and hence derived complete. Since $A^!\simeq \hat\Gamma(Q,W)$, an application of \Cref{lem: coconnective semisimple H^0 = reflexive} proves the desired statement.
\end{proof}

\begin{rmk}
    The proof of \Cref{ginzref} shows that $R^W_3$ is derived complete, and so we obtain an $A_\infty$-quasi-isomorphism $R^W_3\simeq \R\mathrm{End}_{\hat\Gamma(Q,W)}(kQ_0)$. When $W=W^{\geq 4}$, we see that $R^W_3$ is the graded algebra $R_3\cong\mathrm{Ext}^*_{\hat\Pi_3(Q)}(kQ_0,kQ_0)$ equipped with higher $A_\infty$ multiplications. The CY property for $\hat{\Pi}_3(Q)$ yields a description of $R_3$ in terms of $\mathrm{Ext}^*_{kQ}(kQ_0,kQ_0)$ which reduces to the cyclic completion of \Cref{cyccomprmk}. More generally, one should view $R_3^W$ as a `deformed cyclic completion' of $\mathrm{Ext}^*_{kQ}(kQ_0,kQ_0)$.
\end{rmk}

\begin{rmk}
    The condition $W=W^{\geq 3}$ is rather natural and already appears in the literature (e.g.~\cite{HuaZhou}). Indeed, if $W$ has terms of length $1$ or $2$ then one can remove vertices and arrows from $Q$ to obtain a modified quiver with superpotential $(Q',W')$ such that $W'= W'^{\geq 3}$ and $\mathrm{Jac}(W) \simeq \mathrm{Jac}(W')$.
\end{rmk}

\begin{rmk}
As suggested by Matthew Pressland, it would be interesting to establish similar reflexivity results for \textit{relative} Ginzburg algebras \cite{WuIce} and the related class of 1-periodic topological Fukaya categories \cite{Christ}. {This would also establish reflexivity for certain Chekanov--Eliashberg dg algebras \cite{asplund}.}
\end{rmk}

\section{Chains and cochains on topological spaces}
\label{section: Chains and cochains on topological spaces}
In this section we prove that in a wide variety of situations, the dg (co)algebra of (co)chains on a topological space is reflexive. 

\subsection{$\infty$-local systems and Koszul duality}
Here we broadly follow the approach of \cite[Section 5.1]{BDCY}. Fix a field $k$ and let $X$ be a topological space. We define a dg category $\mathcal{C}(X)$ as follows. First view $X$ as a Kan complex via the singular simplicial set functor. Apply the homotopy coherent rigidification functor $\mathfrak{C}$ to the $\infty$-groupoid $X$ to obtain a simplicially enriched category $\mathfrak{C}X$. Linearise the mapping spaces to obtain a category $k\mathfrak{C}X$ enriched in simplicial $k$-vector spaces. Finally, apply the normalised chains functor $N$ of the Dold--Kan correspondence to each mapping space to obtain a dg category $\mathcal{C}(X) \coloneqq N(k\mathfrak{C}X)$. Note that the $\mathcal{C}$ functor is the left adjoint of the dg nerve functor $N_\mathrm{dg}$, so that if $\mathcal{A}$ is a dg category we have a quasi-equivalence $\mathcal{C} N_\mathrm{dg} \mathcal{A}\xrightarrow{\simeq} \mathcal{A}$. The derived category ${\mathcal{D}}(\mathcal{C}X)$ can be interpreted as the category of $\infty$-local systems of dg $k$-vector spaces on $X$.

\begin{rmk}
    The objects of $\mathcal{C}(X)$ are in bijection with the points of $X$, and the mapping spaces are - up to quasi-isomorphism - given by taking $k$-linear chains on the corresponding path space.
\end{rmk}

We let $C_\bullet (X,k)$ denote the dg coalgebra of $k$-chains on $X$, and we let $C^\bullet(X,k)$ denote the dg algebra of $k$-cochains on $X$. We will typically abuse notation and simply denote them as $C_\bullet X$ and $C^\bullet X$, leaving the base field implicit. Observe that $C^\bullet X$ is the $k$-linear dual of $C_\bullet X$. Fixing a point $x\in X$, let $\Omega_x X$ be the space of $x$-based Moore loops in $X$. The `concatenation of loops' operation on $\Omega_xX$ gives the connective dg coalgebra $C_\bullet \Omega_x X$ the structure of a dg algebra.

\begin{rmk}
    If $G_xX$ denotes the $x$-based Kan loop group of the singular simplicial set of $X$, then $C_\bullet(G_xX)$ is a dg Hopf algebra, and there is a quasi-isomorphism $C_\bullet(G_xX)\simeq C_\bullet(\Omega_xX)$ of dg algebras (see e.g.~\cite{GJbook}).
\end{rmk}

If $X$ is path connected then, suppressing the point $x$ from the notation, by \cite[Remark 5.2]{BDCY} we have a quasi-equivalence of dg categories $$\mathcal{C}(X)\simeq C_\bullet \Omega X$$and hence to study $\mathcal{C}(X)$ we may as well study the dg algebra $A\coloneqq C_\bullet\Omega X$. Observe that we have an isomorphism $H^0(A) \cong k\pi_1 (X)$, and in particular $A$ admits a natural augmentation over $k$.

\begin{rmk}
    Although there is a natural map $C_\bullet\Omega X \to k\pi_1(X)$, it generally does not admit a section. Even when $k\pi_1(X)$ is a commutative semisimple $k$-algebra (for example, when $k$ has characteristic zero and $\pi_1(X)$ is finite abelian), the dg algebra $C_\bullet \Omega X$ is rarely a $k\pi_1(X)$-algebra unless $X$ is simply connected.
\end{rmk}

\begin{thm}[{\cite{rzcubes, chlmon}}]\label{rztheorem}Let $X$ be a path connected topological space. Then there is a dg algebra quasi-isomorphism $\Omega (C_\bullet X) \simeq C_\bullet (\Omega X)$.
\end{thm}
\begin{cor}\label{cor: Koszul dual chain and cochains}
Let $X$ be a path connected topological space. Then there is a quasi-isomorphism $(C_\bullet \Omega X)^! \simeq C^\bullet X$.    
\end{cor}
\begin{rmk}
    When $X$ is simply connected, \Cref{rztheorem} reduces to a much earlier theorem of Adams, who introduced the cobar construction $\Omega$ for precisely this reason \cite{adams}.
\end{rmk}
Note that $C^\bullet X$ is coconnective, and has $H^0(C^\bullet X)\simeq k$ if $X$ is path connected.

\begin{ex}
    Take $X=\mathbb{CP}^\infty$, which is a $K(\mathbb{Z},2)$, so that $\Omega X$ is $S^1$. Hence, the dg algebra $A=C_\bullet \Omega X$ is (quasi-isomorphic to) the square-zero extension $k[\epsilon]/\epsilon^2$ with $\epsilon$ in homological degree one. Hence the Koszul dual is $A^!=k[t]$ with $t$ in cohomological degree two. This is indeed the cohomology ring of $X$ (and hence proves that $C^\bullet X$ is formal).
\end{ex}

\begin{ex}\label{circleExample}The following essentially appears as \cite[Example 4.3]{chlmon}. Put $X=S^1$ so that $\Omega X$ is the discrete space $\mathbb{Z}$. So $A\coloneqq C_\bullet \Omega X \simeq k[t,t^{-1}]$ is concentrated in degree zero (the augmentation is given by $t \mapsto 1$). We know that $\Omega C_\bullet X \simeq A$. On the other hand, $C^\bullet X \simeq k[\epsilon] / \epsilon^2$, with $\epsilon$ in cohomological degree one. It follows that $C_\bullet X$ is \textit{quasi-isomorphic} to the dg coalgebra $C$ given by the linear dual of $k[\epsilon] / \epsilon^2$. However, they are \textit{not weakly equivalent}, since their cobar constructions disagree: $\Omega C_\bullet X\simeq A$ while $\Omega C\simeq k[t]$. In fact, $C$ is the coalgebra of chains on the non-grouplike simplicial set $\Delta^1 / \partial\Delta^1$, and the fact that $A$ is the localisation of $C$ at $t$ is a general phenomenon \cite[Corollary 4.4]{chlmon}.
\end{ex}

\begin{rmk}
    If $X$ is simply connected, then the quasi-isomorphism type of the coalgebra $C_\bullet X$ determines its weak equivalence type. If $X$ is not simply connected, this fails, as \Cref{circleExample} demonstrates.
\end{rmk}
\begin{ex}
    Let $G$ be a group. Then $C_\bullet\Omega BG \simeq kG$, and hence $(kG)^!\simeq C^\bullet BG$. In particular, if $G$ is an acyclic group (e.g.\ the Higman group) then $(kG)^!\simeq k$ and in particular $(kG)^{!!}\simeq k$. This gives examples of infinite dimensional discrete algebras with small derived completion.
\end{ex}

\begin{lem}\label{dfdcochainslem}
    Let $X$ be a path connected topological space. Then $\pvd(C^\bullet X)$ is Morita equivalent to $C_\bullet(\Omega X)^{!!}$. Thus, $\pvd(C^\bullet X)$ is a derived completion of $C_{\bullet}(\Omega X)$ in the sense of Efimov \cite{efimovcompletion}.
\end{lem}
\begin{proof}
    By \Cref{cor: Koszul dual chain and cochains} we have a quasi-isomorphism $C_\bullet(\Omega X)^{!}\simeq C^\bullet X$. By \Cref{prop: keller-nicolas locality}, $k$ is a $\pvd$-generator for $C^\bullet X$, so that $\pvd(C^\bullet X)$ is Morita equivalent to $C^\bullet(X)^!$. Combining these results we obtain the desired statement.
\end{proof}

\begin{rmk}
Let $G$ be a group with no nontrivial finite dimensional representations (for example, the trivial group or the infinite alternating group). Let $X$ be any path connected space with $\pi_1(X)\cong G$. The natural t-structure on $D(C_\bullet\Omega X)$ with heart the category of $k\pi_1(X)$-modules shows that $k$ is a thick generator for $\pvd(C_\bullet \Omega X)$, and hence we have $\pvd(C_\bullet \Omega X) \simeq \per(C^\bullet X)$. It follows that $C_\bullet\Omega X$ is reflexive precisely when it is derived complete. 
\end{rmk}

\subsection{Reflexivity results}\label{Section: reflexivity for chains and cochains}

Say that $X$ is \textbf{$k$-finite type} if each $H_i(X,k)$ is finite dimensional over $k$. By the Universal Coefficient Theorem, this is equivalent to $C^\bullet X$ being locally proper. Say that $X$ is \textbf{$k$-finite} if the graded vector space $H_\bullet (X,k)$ is finite dimensional; this is equivalent to $C^\bullet X$ being proper. A finite CW complex is clearly $k$-finite. Observe also that if $X$ is a homotopy retract of a $k$-finite (type) space, then $X$ itself is $k$-finite (type). If $X$ is a $k$-finite type topological space, then it necessarily has finitely many path components.

\begin{lem}\label{serrelem}
    Let $X$ be a simply connected $k$-finite topological space. Then $\Omega X$ is a $k$-finite type space.
\end{lem}
\begin{proof}
It suffices to check that $H^q(\Omega X,k)$ is finite dimensional for every $q$. This is a standard argument using the cohomological Serre spectral sequence associated to the path fibration $\Omega X \to PX \to X$, where $PX\simeq *$ is the path space of $X$. This spectral sequence has $E_2$ page $H^p(X,H^q(\Omega X,k))\cong H^p(X,k)\otimes_k H^q(\Omega X,k)$ and converges to $H^{p+q}(PX,k)$. Since $X$ was $k$-finite, the spectral sequence degenerates after finitely many pages. Since $X$ was simply connected, we have $H^0(\Omega X)\cong k$, and so the $q=0$ column of the $E_2$ page consists of a finite number of finite dimensional vector spaces (namely, the $H^p(X,k)$). Since $PX$ is contractible - and in particular has vanishing $H^1$ - we see that the $q=1$ column of the $E_2$ page must also consist of a finite number of finite dimensional vector spaces. Continuing inductively we see that all entries on the $E_2$ page are finite dimensional, as desired.
\end{proof}

\begin{prop}\label{kfinitespaces}
    Let $X$ be a $k$-finite topological space. Then $C^\bullet X$ is reflexive. If $X$ is simply connected, then $C_\bullet\Omega X$ is reflexive.
\end{prop}
\begin{proof}
The first claim follows from \Cref{propcoconncor}, since $H^0(X)$ is semisimple. For the second claim, note that $C_\bullet \Omega X$ is a connective $k$-algebra, which by \Cref{serrelem} is locally proper. Hence it is reflexive by \Cref{thm: reflexivity via KD for conn. algs}.
\end{proof}

\begin{ex}The simply connected hypothesis of \Cref{kfinitespaces} cannot be dropped. Indeed, take $X=S^1$ from \Cref{circleExample}, so that we have quasi-isomorphisms $C_\bullet \Omega X \simeq k[t,t^{-1}]$ and $C^\bullet X \simeq k[\epsilon] / \epsilon^2$. We compute $(C^\bullet X)^!\simeq k\llbracket t \rrbracket$, which is the completion of $k[t,t^{-1}]$ at the maximal ideal $(t-1)$. In particular $C_\bullet \Omega X$ is not derived complete and so not reflexive (we could have also deduced this from an application of \Cref{noeththm}).
\end{ex}

Say that a topological space $X$ is \textbf{$k\pi_1$-local} if $k\pi_1(X)$ is a finite dimensional local $k$-algebra (i.e.\ a nilpotent extension of $k$). This property will be key for us, since it implies that $k$ is a $\pvd$-generator for $C_\bullet \Omega X$.
\begin{lem}
    Let $X$ be a topological space and $k$ a field. Then $X$ is $k\pi_1$-local if and only if either of the following conditions is satisfied:
    \begin{enumerate}
        \item $X$ is simply connected.
        \item  $k$ has characteristic $p$ and $\pi_1(X)$ is a finite $p$-group.
    \end{enumerate}
\end{lem}
\begin{proof}
This follows from the main theorem of \cite{renault}.
\end{proof}

\begin{thm}\label{kpilocal}
    Let $X$ be a path connected $k\pi_1$-local topological space.
    \begin{enumerate}
    \item There is a natural equivalence $\pvd(C_\bullet \Omega X)\simeq \per(C^\bullet X)$.
    \item If $C_\bullet X$ is reflexive then so is $C^\bullet X$. 
    \item If $\Omega X$ is $k$-finite type, then both $C_\bullet X$ and $C^\bullet X$ are reflexive.
    \item If $C_\bullet X$ is reflexive then there are natural equivalences $$\fd(C_\bullet X)^\op\simeq\pvd(C^\bullet X) \simeq \per(C_\bullet \Omega X).$$
    \end{enumerate}
    
\end{thm}
\begin{proof}
Since $X$ is $k\pi_1$-local, $k$ is a $\pvd$-generator for $C_\bullet \Omega X$ by \Cref{congen}. So putting $A\coloneqq C_\bullet \Omega X$, we have $\pvd(A)\simeq \per(A^!)\simeq \per(C^\bullet X)$ by \Cref{cor: Koszul dual chain and cochains}, which proves claim (1). Claim (2) follows from \Cref{LDreflex}. To prove claim (3), by an application of (2) we need only check that $C_\bullet X$ is reflexive. But this is the same as checking that $\Omega C_\bullet X \simeq C_\bullet \Omega X$ is reflexive, and this follows from \Cref{thm: reflexivity via KD for conn. algs}. Claim (4) follows from (1) combined with \Cref{linduallem}.
\end{proof}
\begin{cor}
    If $X$ is a path connected $k\pi_1$-local space with $\Omega X$ a $k$-finite type space, then every proper $C^\bullet(X)$-module dualises to a compact $C_\bullet(X)$-comodule.
\end{cor}

\begin{rmk}
    The category $\pvd(C_\bullet \Omega X)$ appearing in part (1) of \Cref{kpilocal} is the dg category of $\infty$-local systems of finite dimensional $k$-vector spaces on $X$.
\end{rmk}

\begin{ex}
    Part (3) of \Cref{kpilocal} says that for a pair $(X,k)$ which is of Eilenberg--Moore type in the sense of \cite{DGI}, both $C_\bullet(X,k)$ and $C^\bullet(X,k)$ are reflexive.
\end{ex}

\begin{ex}[Rational homotopy theory]
Suppose that $k$ has characteristic zero and let $X$ be a simply connected $k$-finite type space. Then $C_\bullet \Omega X$ is the universal enveloping algebra of the Whitehead Lie algebra $W(X)\coloneqq \pi_*(\Omega X)\otimes_\mathbb{Z} k$ \cite{FHT}. If $W(X)$ is a $k$-finite type Lie algebra, the PBW theorem tells us that $\Omega X$ is $k$-finite type. In particular, we can conclude that both $C_\bullet X$ and $C^\bullet X$ are reflexive.
\end{ex}

\begin{ex}[$p$-compact groups]Take $k=\mathbb{Z}/p$. Recall that a \textbf{$p$-compact group} in the sense of \cite{DWpcompact} is a triple $(X,BX,e)$ consisting of a $k$-finite space $X$, a pointed $p$-complete space $BX$, and a homotopy equivalence $e:X\to \Omega BX$. We typically refer to $X$ itself as the $p$-compact group. The standard example of a $p$-compact group is the $p$-completion of a compact Lie group $G$ with $\pi_0(G)$ a finite $p$-group. Note that in this setting we have ${(BG)}^{\hat{\phantom{}}}_p \simeq B\hat G_p$ \cite[Proposition 11.9]{DWpcompact}. In particular, a finite $p$-group is a $p$-compact group. If $X$ is a $p$-compact group then $\pi_1BX\simeq \pi_0X$ is finite and hence a $p$-group \cite[Proposition 11.14]{DWpcompact}. It follows from \Cref{kpilocal}(3) that both $C_\bullet BX$ and $C^\bullet BX$ are reflexive. Moreover, $C^\bullet X$ is reflexive by \Cref{kfinitespaces}, and if $\pi_1X\simeq 0$ then $C_\bullet X$ is also reflexive.
\end{ex}

\begin{ex}[String topology]\label{exa: string topology}
Let $M$ be a compact simply connected manifold, so that there is an $S^1$-invariant quasi-isomorphism $\mathrm{HH}_\bullet(C_\bullet \Omega M) \simeq C_\bullet(\mathcal{L}M)$, where $\mathrm{HH}_\bullet$ denotes the Hochschild homology complex and $\mathcal{L}M$ denotes the free loop space of $M$ \cite{jones}. In particular, \Cref{kfinitespaces}, \Cref{kpilocal}(4) and the Morita invariance of Hochschild homology give us a natural quasi-isomorphism $$\mathrm{HH}_\bullet(\pvd(C^\bullet M))\simeq C_\bullet (\mathcal{L}M).$$
Moreover, by \cite{Greflex} and the Morita invariance of Hochschild cohomology \cite{kellerHH} we also obtain natural quasi-isomorphisms $$HH^\bullet(C^\bullet M)\simeq \mathrm{HH}^\bullet(\pvd(C^\bullet M))\simeq \mathrm{HH}^\bullet(C_\bullet \Omega M)$$ and when $M$ is in addition $k$-orientable, Poincar\'e duality (e.g.\ \cite[Example 9.38]{BCLcalabiyau}) gives a quasi-isomorphism $$\mathrm{HH}^\bullet(C^\bullet M)\simeq C_\bullet (\mathcal{L}M)[-\mathrm{dim}M].$$
In characteristic zero, these are quasi-isomorphisms of BV-algebras \cite{tradlerzeinalian}.
\end{ex}

The following is a slight generalisation of \Cref{kpilocal} in the simply connected setting (recall that if $\pi_1(X)\cong 0$ then $H^1(X)\cong 0$ by the Hurewicz and Universal Coefficient Theorems). The proof is a straightforward application of \Cref{thm: reflexivity via KD for coconn. algs}.

\begin{prop}
Let $X$ be a path connected $k$-finite type topological space such that $H^1(X)$ vanishes. Then the dg algebra $C^\bullet X$ is reflexive.
\end{prop}
\begin{cor}
    Let $G$ be a finite perfect group. Then $C^\bullet BG$ is reflexive.
\end{cor}

\begin{ex}[Fukaya categories of cotangent bundles]\label{rmk: Fukaya category cotangent bundle}
Let $M$ be a compact path connected smooth manifold. In this setting, $C_\bullet \Omega M $ is Morita equivalent to $\mathcal{W}(T^*M)$, the derived wrapped Fukaya category of the symplectic manifold $T^*M$. If $M$ is $k\pi_1$-local, then using \Cref{kpilocal} we obtain an equivalence $\pvd(\mathcal{W}(T^*M))\simeq \per(C^\bullet M)$. When $M$ is in addition simply connected, then $C^\bullet M$ is Morita equivalent to $\mathcal{F}(T^*M)$, the compact Fukaya category of $T^*M$, and we hence obtain equivalences $$\begin{array}{ccc}\per(\mathcal{W}(T^*M))\simeq \pvd(\mathcal{F}(T^*M)) & \text{and} & \pvd(\mathcal{W}(T^*M))\simeq \per(\mathcal{F}(T^*M)).\end{array}$$Versions of the above equivalences for more general simply connected symplectic manifolds were given in \cite{EkholmLekili} and versions for Milnor fibres were given in \cite{LekiliUeda}. We will see more about Fukaya categories in \Cref{gentlesection}. 
\end{ex}

\section{Gluing reflexive dg categories}

In this section we note that one can glue (semi)reflexive dg categories along semiorthogonal decompositions. Semiorthogonal decompositions, introduced in \cite{bondalkapranov}, are a fundamental tool in derived noncommutative algebraic geometry, as they allow one to decompose invariants and to isolate singular behaviour. A survey of their uses can be found in \cite{kuzsurvey}.

\begin{defn}\label{def: semiorthogonal decomposition}
Let $\mathcal{T}$ be a triangulated category. A \textbf{semiorthogonal decomposition}  of $\mathcal{T}$ is a pair of thick subcategories $\mathcal{A},\mathcal{B} \subseteq \mathcal{T}$ satisfying the following properties:
\begin{enumerate}
\item The smallest thick subcategory of $\mathcal{T}$ containing $\mathcal{A}$ and $\mathcal{B}$ is $\mathcal{T}$ itself.

\item $\Hom_{\mathcal{T}}(b,a) = 0$ for all $a \in \mathcal{A}$ and $b \in \mathcal{B}$. 

\end{enumerate}
In this situation we write $\mathcal{T} = \langle \mathcal{A},\mathcal{B} \rangle$.
\end{defn}
If $\mathcal{T}$ is a pretriangulated dg category, we say that a semiorthogonal decomposition of $\mathcal{T}$ is a semiorthogonal decomposition of the triangulated category $H^0(\mathcal{T})$.

\begin{rmk}
Semiorthogonal decompositions were studied in the context of reflexivity in \cite{KS}. If $\mathcal{T}$ is a pretriangulated dg category with a semiorthogonal decomposition $\mathcal{T} = \langle\mathcal{A}, \mathcal{B} \rangle$ then there is a semiorthogonal decomposition $\pvd(\mathcal{T}) = \langle\pvd\mathcal{B}, \pvd\mathcal{A} \rangle$ \cite[Lemma 3.7]{KS}. Moreover, if $\mathcal{T}$ is reflexive then so are $\mathcal{A}$ and $\mathcal{B}$; this follows from naturality of the evaluation functor. Alternatively, this can be seen using the monoidal characterisation of \cite{Greflex}, using the fact that reflexive objects are closed under retracts. 
\end{rmk}

\begin{thm}\label{thm:Gluing over Semiorthogonal decompositions}
Let $\mathcal{\mathcal{T}}$ be a semireflexive dg category that admits a semiorthogonal decomposition $\mathcal{D}^{\perf}(\mathcal{T}) = \langle \mathcal{A}, \mathcal{B} \rangle$. If $\mathcal{A}$ and $\mathcal{B}$ are reflexive then so is $\mathcal{T}$.
\end{thm}

\begin{proof}
In this proof, to avoid opposite categories appearing we use the coevaluation functor instead of the evaluation functor. Recall that $\coev_{\mathcal{A}}$ is defined in \cite{KS} as the composition of $\ev_{\mathcal{A}}$ with the linear dual functor $(-)^\ast$. Moreover, (semi)reflexivity can be checked using the coevaluation functor in exactly the same manner as the evaluation functor, cf.~\cite[Lemma 3.10]{KS}.

Without loss of generality we may assume that $\mathcal{T}$ is pretriangulated and idempotent complete, so that $\mathcal{T} = \mathcal{D}^{\perf}(\mathcal{T})$. Applying \cite[Lemma 3.7]{KS} twice, we obtain a semiorthogonal decomposition
\[
\pvd\pvd(\mathcal{T}) = \langle \pvd\pvd(\mathcal{A}), \pvd\pvd(\mathcal{B}) \rangle
\]
 It follows from the proof of \cite[Lemma 3.7]{KS} that the semiorthogonal decompositions are compatible with the (co)evaluation functors, in the sense that the following diagram commutes:

\[
\begin{tikzcd}
\mathcal{A} \arrow[r,"\coev_{\mathcal{A}}"] \arrow[d,"i_{\mathcal{A}}",shift left] & \pvd\left(\pvd(\mathcal{A}) \right)\arrow[d,"i'_{\cA}"] \\
\mathcal{T} \arrow[r,"\coev_{\mathcal{T}}"] & \pvd\left(\pvd(\mathcal{T}) \right).
\end{tikzcd}
\]
Here $i_{\mathcal{A}}$ and $i_{\cA}^{\prime}$ denote the inclusions. Note that these functors have left adjoints, which we denote by $\pi_{\mathcal{A}}$ and $\pi'_{\mathcal{A}}$ respectively. By assumption $\mathcal{T}$ is semireflexive (i.e.\ $\coev_{\mathcal{T}}$ is fully faithful) and so to prove that it is reflexive we need only check that $\coev_{\mathcal{T}}$ is essentially surjective. Take $M \in \pvd\pvd(\mathcal{T})$. Then there is an exact triangle
\[
i_{\cB}' \pi_{\cB}' M \to M \to i_{\cA}' \pi_{\cA}' M \to
\]
Since $\mathcal{A}$ is reflexive we have $\pi'_{\mathcal{A}} M \simeq \coev_{\mathcal{A}}(a)$ for some $a \in \mathcal{A}$. So we have
\[
i_{\cA}' \pi_{\cA}' M \simeq i_{\cA}' \coev_{\mathcal{A}}(a) \simeq \coev_{\mathcal{T}} i_{\mathcal{A}}(a)
\]
Similarly, $
i'_{\mathcal{B}} \pi'_{\mathcal{B}} M \simeq \coev_{\mathcal{T}} i_{\mathcal{B}}(b)$
for some $b \in \mathcal{B}$, and since $\mathcal{T}$ is semireflexive the morphism 
\[
 \coev_{\mathcal{T}} i_{\mathcal{A}}(a)[-1] \to \coev_{\mathcal{T}} i_{\mathcal{B}}(b)
\]
in the triangle above can be lifted to some $f\colon i_{\mathcal{A}}(a)[-1] \to  i_{\mathcal{B}}(b)$. Therefore it follows that $\coev_{\mathcal{T}}(\mathrm{cone}(f)) \simeq M$, as required. 
\end{proof}

\begin{rmk}Examples of semireflexive dg categories include proper dg categories (\Cref{refexamples}) and dg algebras which are derived complete with respect to a $\pvd$-generator (\Cref{semirmk}).
\end{rmk}

\begin{rmk}
    If $\mathcal{A}$, $\mathcal{B}$ are two dg categories and $M$ is a $\mathcal{B}$-$\mathcal{A}$-bimodule, then one can glue $\mathcal{A}$ and $\mathcal{B}$ along $M$ to produce a new dg category $\mathcal{A} \xrightarrow{M} \mathcal{B}$. The objects of the gluing are $\mathrm{Ob}(\mathcal{A}) \sqcup \mathrm{Ob}(\mathcal{B})$, and the morphisms are given by upper-triangular matrices. Even if both $\mathcal{A}$ and $\mathcal{B}$ are pretriangulated, then usually $\mathcal{A} \xrightarrow{M} \mathcal{B}$ will fail to be; in this situation we let $\mathcal{A}\vdash_M \mathcal{B} \coloneqq \per (\mathcal{A} \xrightarrow{M} \mathcal{B})$. Orlov shows that the homotopy category of $\mathcal{A}\vdash_M \mathcal{B}$ admits a semiorthogonal decomposition $\langle \mathcal{A},\mathcal{B}\rangle$, and moreover if $\mathcal{T}= \langle \mathcal{A},\mathcal{B}\rangle$ is any pretriangulated dg category then there is a bimodule $M$ and a quasi-equivalence $\mathcal{T}\simeq \mathcal{A} \vdash_{M} \mathcal{B}$ \cite{orlovgluing}. Note that $\mathcal{A} \vdash_{M} \mathcal{B}$ is reflexive exactly when the gluing $\mathcal{A} \xrightarrow{M} \mathcal{B}$ is reflexive.
\end{rmk}

\section{Fukaya categories and graded gentle algebras}\label{gentlesection}
 We apply the techniques of this paper to study reflexivity for partially wrapped Fukaya categories of surfaces in the sense of \cite{HaidenKatzarkovKontsevich} and the related class of graded gentle algebras.

\subsection{Topological notions}
By a \textbf{marked surface} we mean a pair $(\Sigma, \cM)$, where $\Sigma$ is a compact, oriented smooth surface with boundary $\partial \Sigma \neq \emptyset$ and where $\cM \subseteq \partial \Sigma$ is a compact subset which intersects each boundary component non-trivially and such that each connected component of $\cM$ has non-trivial interior. We often omit the set $\cM$  from the notation and refer to a marked surface as simply $\Sigma$. Boundary components $B \subseteq \partial \Sigma$ such that $B \subseteq \cM$ are called \textbf{fully marked} while all other connected components of $\cM$, namely those homeomorphic to $[0,1]$, are called \textbf{marked intervals}. As consequence of this definition, $\cM$ is the union of the fully marked components and the marked intervals. By a \textbf{boundary segment} or \textbf{stop} we mean the closure in $\Sigma$ of a connected component of $\partial \Sigma \setminus \cM$.

A \textbf{simple arc} on a graded surface $(\Sigma, \cM)$
is an embedded path $\gamma\colon [0,1] \hookrightarrow \Sigma$ such that $\gamma^{-1}(\cM)=\{0,1\}$. Such an arc is said to be \textbf{finite} if both its end points lie in marked intervals. By an \textbf{isotopy} between simple arcs $\gamma_0$ and $\gamma_1$ we mean a map $H\colon [0,1] \times [0,1] \rightarrow \Sigma$ such that for all $t \in [0,1]$, $H_t\coloneqq H|_{\{t\} \times [0,1]}$ is a simple arc and such that $H^{-1}(\cM)=[0,1] \times \{0,1\}$ as well as $H_0=\gamma_0$ and $H_1=\gamma_1$.

An \textbf{arc system} $\Gamma$ is a collection of pairwise disjoint and pairwise non-isotopic simple arcs on $\Sigma$. It follows that $\Gamma$ is necessarily a finite set. An arc system $\Gamma$ is \textbf{full} if its complement $$\Sigma \setminus \Gamma \coloneqq \Sigma \setminus \big(\bigcup_{\gamma \in \Gamma}\gamma([0,1])\big)$$ is homeomorphic to a disjoint union of disks, the boundary of each of which contains at most one boundary segment. Likewise, $\Gamma$ is \textbf{finitely-full} if its complement is homeomorphic to a disjoint union of disks as above as well as half-open cylinders $C=[0,1) \times S^1$ such that $\partial C=\{0\} \times S^1$ corresponds to a fully marked component. A full (resp.~finitely-full) arc system is \textbf{formal} if the boundary of every disk component of $\Sigma \setminus \Gamma$ contains a (and hence exactly one) boundary segment. 

Of course, full formal arc systems can only exist on surfaces with at least one marked interval  but in fact, this is also a sufficient condition.
\begin{prop}[{\cite[]{HaidenKatzarkovKontsevich}}]
If $\Sigma$ contains at least one marked interval, then $\Sigma$ admits a full, formal arc system.
\end{prop}
 We note that the same is true for finitely-full, formal arc systems under the same assumptions; the proof is similar.

 A \textbf{flow} from an arc $\gamma_1 \in \Gamma$ to a (possibly identical) arc $\gamma_2 \in \Gamma$ is the homotopy class of a path $f$ inside $\cM$ which follows the natural orientation of $\partial \Sigma$ and which starts on a start or end point of $\gamma_1$ and which ends on a start or end point of $\gamma_2$. A flow $f$ is \textbf{constant} if it agrees with the homotopy class of a constant path. Flows can be composed as paths and a flow is \textbf{irreducible} if it cannot be  decomposed further into non-constant flows between arcs of $\Gamma$. As such, every flow between arcs of $\Gamma$ is a finite composition of irreducible ones in a unique way.

\subsection{Partially wrapped Fukaya categories and gentle algebras}

By a \textbf{graded marked surface} we mean a triple $(\Sigma, \cM, \eta)$ consisting of a graded surface $(\Sigma, \cM)$ and a \textbf{line field} $\eta$ on $\Sigma$, that is, a section $\eta\colon \Sigma \rightarrow \mathbb{P}(T\Sigma)$ of the projectivised tangent bundle. As with $\cM$, the line field is frequently omitted from our notation, so that the graded marked surface above is simply referred to by $\Sigma$. The presence of a line field allows one to endow every simple arc on $\Sigma$ with the extra structure of a grading, cf.~\cite{HaidenKatzarkovKontsevich}. The set of all gradings on every simple arc is (non-canonically) in bijection with $\mathbb{Z}$ and given two graded simple arcs $\gamma_1, \gamma_2$ in an arc system, any flow $f$ from $\gamma_1$ to $\gamma_2$ naturally inherits a degree $|f| \in \mathbb{Z}$ which is compatible with the composition of flows, cf.~\cite[(3.17)]{HaidenKatzarkovKontsevich}. We denote by $\Flow(\gamma_1, \gamma_2)$ the set of flows from $\gamma_1$ to $\gamma_2$.

\begin{defn}[{\cite[]{HaidenKatzarkovKontsevich}}]\label{def: Fukaya category arc system} Let $\Sigma$ be a graded marked surface with at least one marked interval and let $\Gamma \subseteq \Sigma$ be a graded full (resp.~finitely-full) formal arc system. Define $\cF=\cF(\Gamma)$ as the $k$-linear graded category with $\Ob{\cF}=\Gamma$ and morphism spaces $$\Hom_{\cF}^{\bullet}(\gamma_1, \gamma_2)\coloneqq k\Flow(\gamma_1, \gamma_2),$$
 for all $\gamma_1, \gamma_2$, where each $f \in \Flow(\gamma_1, \gamma_2)$ is regarded as a homogeneous element of degree $|f|$. The composition law for morphisms is the unique $k$-linear extension of the composition of flows so that $g \circ f=0$ whenever the start point of the flow $g$ does not agree with the end point of $f$. The \textbf{partially wrapped Fukaya category} of $\Sigma$ is the category $\Fuk(\Sigma)\coloneqq \per(\cF(\mathcal{E}))$, {where $\mathcal{E} \subseteq \Sigma$ denotes any full formal arc system. Likewise, we define $\Fuk^f(\Sigma)\coloneqq \per(\cF(\mathcal{A}))$ for any finitely-full formal arc system $\mathcal{A}$. Finally, we set $\Fuk^{\operatorname{inf}}(\Sigma) \coloneqq \pvd(\Fuk(\Sigma))$.}
\end{defn}
\begin{rmk}
{By  \cite[Remark 3.21]{LekiliPolishchuk} and its preceding paragraph, $\Fuk^{\operatorname{inf}}(\Sigma)$ is a model for the \textit{infinitesimal Fukaya category}. Since $\Fuk(\Sigma)$ is homologically smooth, $\Fuk^{\operatorname{inf}}(\Sigma)$ is naturally identified with a full subcategory of $\Fuk(\Sigma)$. If $\Sigma$ does not contain components without stops and vanishing winding number, it follows from \Cref{prop: Koszul duality gentle on categories} below that $\Fuk^f(\Sigma)=\Fuk^{\operatorname{inf}}(\Sigma)$. In general, $\Fuk^f(\Sigma)$ is a derived completion of $\Fuk(\Sigma)$.}
\end{rmk}

\begin{ex}\label{exa: polynomial ring as Fukaya category}
Let $n \in \mathbb{Z}$ and $A=k[t]$ with $|t|=n$. Then $A=\cF(\{\gamma\})$, where $\gamma \subseteq \Sigma$ is any embedded arc on an annulus with one fully marked component and one marked interval which connects the fully marked component and the marked interval. The degree of $|t|$ is controlled by the line field.
\end{ex}

 Although a priori $\Fuk(\Sigma)$ depends on the choice of an arc system, it turns out that its Morita equivalence class is well-defined, cf.~\Cref{fukmorita} below. 

We note that $\cF(\Gamma)$ is canonically augmented: for each object $\gamma \in \Gamma=\Ob{\cF(\Gamma)}$, the augmentation map is the projection $k \Flow(\gamma, \gamma) \rightarrow k$ onto the basis element of the constant flow. In particular, this induces a canonical augmentation over a product of fields on the category algebra of $\cF(\Gamma)$, that is, the algebra $$\bigoplus_{\gamma, \gamma' \in \Gamma} k \Flow(\gamma, \gamma'),$$
obtained by taking the (finite) direct sum of all morphism spaces in $\cF(\Gamma)$.

The following proposition recalls a few useful properties of the previous constructions. It follows from \cite[Proposition 3.5]{HaidenKatzarkovKontsevich} and \Cref{def: Fukaya category arc system}.

\begin{prop}\label{prop: Fukaya category properness  smoothness} Let $\Gamma$ be a graded full or finitely-full formal arc system on a graded marked surface $\Sigma$ with at least one marked interval. Then the following hold:

\begin{enumerate}
    \item $\cF(\Gamma)$ is proper if and only if $\Gamma$ is finitely-full;
    \item  
    $\cF(\Gamma)$ is a smooth dg category if and only if $\Gamma$
    is full; in particular $\Fuk(\Sigma)$ is smooth.
    \item $\Fuk(\Sigma)$ is smooth and proper if and only if $\Sigma$ has no fully marked components. Equivalently, this is the case if and only if $\Sigma$ admits a formal arc system which is simultaneously full \emph{and} finitely-full.
\end{enumerate}
\end{prop}

 The category algebra of $\cF(\Gamma)$ can be described by a quiver with quadratic relations and vertex set $\Gamma$ whose set of arrows is given by the set of irreducible flows. The set of relations consists of all paths $g f$ of length $2$ such that $g \circ f=0$ in $\cF(\Gamma)$. Algebras of this kind are graded and possibly infinite dimensional analogoues of \textbf{gentle algebras} which were first introduced in \cite{AssemSkowronski}. Every smooth graded gentle algebra $A$ arises as the category algebra of $\cF(\Gamma)$ for some full formal arc system $\Gamma$ on a graded marked surface $\Sigma_A$. Likewise, every proper graded gentle algebra $B$ is the category algebra of $\cF(\Gamma)$ for some finitely-full formal arc system $\Gamma$ on a surface $\Sigma_B$. For an explicit construction of the surfaces associated to a graded gentle algebra we refer the reader to \cite{HaidenKatzarkovKontsevich, OpperPlamondonSchroll, LekiliPolishchuk}. The definition of $\cF$  can be extended to non-formal arc systems which are full or finitely-full \cite{HaidenKatzarkovKontsevich} and hence $\Fuk(\Sigma)$ can be defined for all graded marked surfaces $\Sigma$. In the non-formal case, $\cF(\Gamma)$ is a minimal $A_\infty$-category whose underlying graded $k$-linear category is constructed in the same way as in the formal case. However, in addition to composition,  higher $A_\infty$-compositions are contributed by those disks in $\Sigma \setminus \Gamma$ which do \emph{not} contain a boundary segment, or in other words, disks which are bounded by alternating sequences of marked intervals and arcs of $\Gamma$. For a formal arc system $\Gamma$ the complement $\Sigma \setminus \Gamma$ does not contain such disks, and the $A_\infty$-category $\cF(\Gamma)$ is formal and reduces to the one of \Cref{def: Fukaya category arc system}. Next, we recall an important property of the construction $\Gamma \mapsto \cF(\Gamma)$.

 \begin{prop}[{\cite[Lemma 3.2]{HaidenKatzarkovKontsevich}}]\label{prop: Inclusion = Morita equivalence}
 Let $\Gamma \subseteq \Gamma' \subseteq \Sigma$ be full (resp.~finitely-full) arc systems. Then the canonical inclusion $\cF(\Gamma) \hookrightarrow \cF(\Gamma')$ is a Morita equivalence, that is, the induced $A_\infty$-functor $\per(\cF(\Gamma)) \rightarrow \per(\cF(\Gamma'))$ is an equivalence of triangulated categories.
 \end{prop}
 We note that while the previous proposition appears in \cite{HaidenKatzarkovKontsevich} only for full arc systems, its proof extends immediately to the finitely-full case. It implies the following.
\begin{prop}\label{fukmorita}
Let $\Gamma, \Gamma' \subseteq \Sigma$ be full (resp.~finitely-full) arc systems. Then $\cF(\Gamma)$ and $\cF(\Gamma')$ are Morita equivalent.
\end{prop}
\begin{proof}For full arc systems, this was proved in \cite{HaidenKatzarkovKontsevich}. It relies on the fact that the category of full arc systems, with morphisms given by inclusions, is contractible, which in turn is a consequence of the contractibility of the arc complex of $\Sigma$. Likewise, \cite[Corollary]{HatcherTriangulations} shows that every two inclusion maximal arc systems on a graded marked surface (``triangulations'') are connected by a sequence of arc inclusions and removals, or in other words, a zig-zag in the category of arc systems. Since every arc system sits inside an inclusion maximal one, this shows that any two full arc systems can be connected through a zig-zag of inclusions. In fact, the result in  \textit{op.~cit.}\ does not require the set of markings $\cM$ to intersect every boundary component, and when applied to the ``marked'' surface obtained by removing all fully marked components from $\cM$, the previous arguments imply that any two finitely-full arc systems  are connected by a zig-zag of inclusions passing through finitely-full arc systems. Now \Cref{prop: Inclusion = Morita equivalence} implies that $\cF(\Gamma)$ are $\cF(\Gamma')$ are Morita equivalent.
\end{proof}

\subsection{Koszul duality between smooth and proper gentle algebras}\label{SectionKoszulDuality}

 We recall the precise relationship between proper and  smooth graded gentle algebras. In what follows $\Sigma$ denotes a graded marked surface. 

Suppose that  $\Gamma \subseteq \Sigma$ is a graded full formal arc system. We construct a dual finitely-full arc system $\Gamma^! \subseteq \Sigma$ from $\Gamma$ as follows. For every arc $\gamma \in \Gamma$, denote by $\gamma^{\bot}$ any arc which intersects $\gamma$ transversally in the interior but no other arc $\delta \in \Gamma \setminus \{\gamma\}$ and whose end points lie in the unique boundary segment of the two disks of $\Sigma \setminus \Gamma$ which border $\gamma$ from either side, see \Cref{fig:koszul dual arc system}. Then let $\gamma^!$ denote any simple arc obtained from $\gamma^{\bot}$ by sliding both its end points along $\partial \Sigma$, following the induced orientation of $\partial \Sigma$, until the end points land in the marked interval next to the respective boundary segment. The collection $\Gamma^!\coloneqq \{\gamma^! \mid \gamma \in \Gamma\}$ is well defined up to isotopy. By deforming $\gamma^!$ if necessary, we can arrange that $\Gamma^!$ is an arc system (which is then finitely-full) and that for all $\delta \in \Gamma$, $\gamma^!$ and $\delta$ are in minimal position, that is, all their intersections are transversal and the number of their intersections is minimal within their respective homotopy classes.

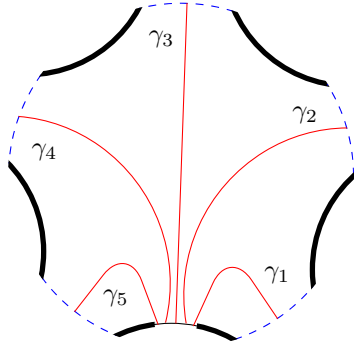
\begin{figure}

\begin{tikzpicture}[scale=2.3]
\def\r{1} 
\def\ratio{0.6} 
\def\ratioangles{1.25}
\def\N{5} 

\pgfmathsetmacro{\variable}{-1}; 

\def\rtwo{\r * \ratio} 

\pgfmathsetmacro{\betaangle}{360/ ((\N) * (1+(\ratioangles)))}

\pgfmathsetmacro{\alphaangle}{\betaangle * \ratioangles}

\clip (0,0) circle[radius=\r];

\pgfmathsetmacro{\h}{sqrt(2 * \r^2 (1- cos(\betaangle)))}
\pgfmathsetmacro{\distance}{sqrt(\r^2 -(\h/2)^2)+ sqrt(\rtwo^2 -(\h/2)^2)} 

\pgfmathsetmacro{\rtwo}{sqrt(\distance^2 + \r^2 - 2*cos(\alphaangle/2) * \distance *\r}

\coordinate (base) at ({-(-\distance+\rtwo)*cos(\variable*(\betaangle+\alphaangle)-\alphaangle/2)},{-(-\distance+\rtwo)*sin(\variable*(\betaangle+\alphaangle)-\alphaangle/2)});

\foreach \x in {1,..., \N} 
{

 \draw[blue, dashed, thick] ({\r*cos(\x*(\betaangle+\alphaangle))},{\r*sin(\x*(\betaangle+\alphaangle))}) arc ({\x*(\betaangle+\alphaangle)}:{\x*(\betaangle+\alphaangle)+\betaangle}:\r);

 \begin{scope}[shift={({\distance*cos(\x*(\betaangle+\alphaangle)-\alphaangle/2)},{\distance*sin(\x*(\betaangle+\alphaangle)-\alphaangle/2)})}]
 \draw[line width=2pt] (0,0) circle (\rtwo);  
 \end{scope}
}


\filldraw[white] (base) circle (3.4pt);

 \draw[] ({\distance*cos(\variable*(\betaangle+\alphaangle)-\alphaangle/2)},{\distance*sin(\variable*(\betaangle+\alphaangle)-\alphaangle/2)}) circle (\rtwo);

\pgfmathsetmacro{\offset}{\alphaangle/(2*\N)}
\foreach \y in {2,..., 4} 
{

\pgfmathsetmacro{\x}{\y-2};
\pgfmathsetmacro{\aux}{\distance-\rtwo};
\draw[hobby, red] plot [tension=1, in=90] coordinates{ ({\r*cos(\x*(\betaangle+\alphaangle)+\betaangle/2)},{\r*sin(\x*(\betaangle+\alphaangle)+\betaangle/2)}) ({\r/2*cos(\x*(\betaangle+\alphaangle)+\betaangle/2)},{\r/2*sin(\x*(\betaangle+\alphaangle)+\betaangle/2)}) ({\aux*cos(\variable*(\betaangle+\alphaangle)-\alphaangle/2-\offset*(\x-1))},{\aux*sin(\variable*(\betaangle+\alphaangle)-\alphaangle/2-\offset*(\x-1))) }) } ;

 \draw ({\r*0.8*cos(\x*(\betaangle+\alphaangle)+\betaangle/2+10)},{\r*0.8*sin(\x*(\betaangle+\alphaangle)+\betaangle/2+10)}) node{$\gamma_{\y}$};
}

\pgfmathsetmacro{\offset}{\alphaangle/(2.4*\N)}
 \foreach \y in {1,5} 
{

\pgfmathsetmacro{\x}{\y-2};
\pgfmathsetmacro{\aux}{\distance-\rtwo+0.018};
\draw[rounded corners=15pt, red] ({\r*cos(\x*(\betaangle+\alphaangle)+\betaangle/2)},{\r*sin(\x*(\betaangle+\alphaangle)+\betaangle/2)})--({\r/2*cos(\x*(\betaangle+\alphaangle)+\betaangle/2)},{\r/2*sin(\x*(\betaangle+\alphaangle)+\betaangle/2)})--({\aux*cos(\variable*(\betaangle+\alphaangle)-\alphaangle/2-\offset*(\x-1))},{\aux*sin(\variable*(\betaangle+\alphaangle)-\alphaangle/2-\offset*(\x-1))) });

 \draw ({\r*0.8*cos(\x*(\betaangle+\alphaangle)+\betaangle/2+10)},{\r*0.8*sin(\x*(\betaangle+\alphaangle)+\betaangle/2+10)}) node{$\gamma_{\y}$};
}

\end{tikzpicture}
    \caption{Marked intervals (black solid) and arcs $\gamma_1, \dots, \gamma_5$ (blue dashed) and $\gamma_1^{\bot}, \dots, \gamma_5^{\bot}$ (red solid). The duals $\gamma_i^!$ are obtained by sliding the end points of the arcs $\gamma_i^{\bot}$ on the boundary segment onto the neighbouring marked interval to their right while keeping the linear order among each other.}
    \label{fig:koszul dual arc system}
\end{figure}

Note that by construction, either $|\Flow(\gamma, \gamma^!)|=1$ or $\gamma$ and $\gamma^!$ intersect in a single (interior) point. On the other hand, for all other $\delta \in \Gamma \setminus \{\gamma\}$, $\delta$ and $\gamma^!$ are disjoint and $\Flow(\delta, \gamma^!)=\emptyset$. In fact these properties uniquely characterise the isotopy classes in the collection $\Gamma^!$. Because each step of the process of passing from $\gamma$ to $\gamma^!$ is reversible, the assignment $\Gamma \mapsto \Gamma^!$ defines a bijection between the sets of full formal arc systems and finitely-full and formal arc systems on $\Sigma$ up to isotopy.

Although it will not be very important in this paper, $\gamma^!$ inherits a canonical grading from $\gamma$ by requiring that the unique flow $f \in \Flow(\gamma, \gamma^!)$ or the unique intersection of $\gamma$ and $\gamma^!$ is of degree $0$. Moreover, with the given grading, there is a bijection 
\begin{displaymath}
\begin{tikzcd}[row sep=0.5em]
\Flow(\gamma_1, \gamma_2) \arrow{r}{\sim} & \Flow(\gamma_2^!, \gamma_1^!), \\
f \arrow[mapsto]{r} & f^!
\end{tikzcd}
\end{displaymath}
such that $|f^!|=1-|f|$. This arc system duality is a topological incarnation of Koszul duality:
\begin{prop}[{\cite[Appendix C]{OpperPlamondonSchroll}}]\label{prop: Koszul dual of Fukaya category} Let $\Gamma \subseteq \Sigma$ be a full formal arc system and let $A$  denote the category algebra of $\cF(\Gamma)$ endowed with its canonical augmentation.  Then $A^!$ is quasi-isomorphic to the category algebra of $\cF(\Gamma^!)$ with $\Gamma^!$ endowed with the induced grading.
\end{prop}

\begin{ex}\label{exa: Koszul dual polynomial rings}
Let $A$ and $\Sigma$ be as in \Cref{exa: polynomial ring as Fukaya category}. Then $A^!$ is quasi-isomorphic to $k[x]/(x^2)$ with $|x|=1-|t|$. The corresponding arc system of $A^!$ on the annulus is the embedded arc which starts and ends on the marked interval and which traverses the fully marked component once.
\end{ex}
 In particular, Koszul duality induces a bijection between the sets of isomorphism classes of smooth graded gentle algebras and the set of isomorphism classes of proper graded gentle algebras. To state the precise Koszul duality on the level of categories, we recall that every embedded, oriented closed curve $\gamma$ on $\Sigma$ inherits a \textbf{winding number} $\omega(\gamma) \in \mathbb{Z}$ thanks to the line field $\eta$, cf.~\cite{LekiliPolishchuk}. The winding number of a boundary component $B \subseteq \partial \Sigma$ is by definition  the winding number of the closed curve corresponding to the \emph{orientation preserving} identification $S^1 \cong B \subset \partial \Sigma$. Here, $B$ is endowed with the induced orientation which it inherits from the orientation of $\Sigma$ in the usual way. If $B$ is a fully marked component, $\omega(B)$  has a particularly simply algebraic interpretation: for any arc $\gamma$ with an endpoint $p$ on $B$, $\omega(B)$ coincides with the degree of the flow which starts and ends at $p$ and traverses $B$ exactly once. On the level of categories, Koszul duality gives us the following relationship.

\begin{prop}[{\cite[Appendix C, Step 6]{OpperPlamondonSchroll}}]\label{prop: Koszul duality gentle on categories} Let $\Gamma \subseteq \Sigma$ be a full formal graded arc system and $\Gamma^! \subseteq \Sigma$ its dual equipped with any grading. Let further $\cT \subseteq \cD\big(\cF(\Gamma^!)\big)$ denote the thick subcategory associated to the augmentation of ${\cF(\Gamma)}^!$, that is, $\cT \simeq \per({\cF(\Gamma)}^{!!})$. Then there exists an essentially surjective exact functor
\begin{displaymath}
\begin{tikzcd}
\Fuk(\Sigma)\simeq \per(\cF(\Gamma)) \arrow{r} & \cT.
\end{tikzcd}
\end{displaymath}
which is an equivalence if and only if $\Sigma$ contains no fully marked components with vanishing winding number. Moreover, this is the case if and only if $\cF(\Gamma)$ is derived complete.
\end{prop}

 In fact, as shown in \cite{OpperPlamondonSchroll}, the functor in the previous proposition is equivalent to the derived completion functor $\per(\cF(\Gamma)) \rightarrow \per(\cF(\Gamma)^{!!})$. {We will see below that that the codomain of this completion functor is equivalent to $\pvd(\Fuk^f(\Sigma))$. In particular, $\Fuk(\Sigma) \simeq \pvd(\Fuk^f(\Sigma))$ if $\Sigma$ contains no fully marked components with vanishing winding number.}

\begin{cor}\label{cor: Fukaya category as proper graded gentle}
If $\Sigma$ contains no fully marked components with vanishing winding number then $$\Fuk(\Sigma) \simeq \thick_A A/\rad(A)$$
for some proper graded gentle algebra $A$ with graded marked surface $\Sigma_A\cong \Sigma$.
\end{cor}

\subsection{Semiorthogonal decompositions, reflexivity and $\pvd$-generators} \label{section: reflexivity proper generation gentle}

\begin{prop}\label{prop: semiorthogonal decomposition gentle}
Let $A$ be a proper graded gentle algebra and let $\{B_1, \dots, B_n\} \subset \partial \Sigma_A$ denote its set of fully marked components. Then there exists a semiorthogonal decomposition of the form $$\per(A) \simeq \langle \per(A^{(n)}), \per(k[x_1]/(x_1^2)), \dots, \per(k[x_n]/(x_n^2)) \rangle,$$
where $|x_i|=1-\omega(B_i)$ and $A^{(n)}$ is a smooth proper graded gentle algebra.
\end{prop}
\begin{proof}
Set $\Sigma=\Sigma_A$. Because $A$ is proper, there exists a marked interval $I \subseteq \partial \Sigma_A$. Now choose a simple arc $\gamma_n\colon[0,1] \rightarrow \Sigma_A$ which starts in $I$ and ends on $B_n$. Then by concatenating $\gamma_n$ first with the closed embedded boundary path with endpoints $\gamma_n(1)$  which traverses $B_n$ in counter-clockwise direction and then with the inverse of $\gamma_n$, we obtain a path $\delta_n$ which can be homotoped relative to $\partial \Sigma$ into a simple arc with endpoints in $x,y \in I$. The natural orientation of $I \subseteq \partial \Sigma$ then induces a total order on the set $\{x,y\}$, say $x < y$. The complement of $\delta_n$ in $\Sigma$ has two connected components, one of which, which we shall call $U$, is a half-open cylinder $C\cong [0,1) \times S^1$  whose boundary $\{0\} \times S^1$ coincides with $B_n$. After smoothing out the complement of $\delta_n$, we obtain a graded surface $\Sigma^{(1)}\coloneqq \Sigma \setminus \delta_n \subset \Sigma$ whose line field is obtained by restriction of $\eta_A$. We endow it with the set of markings obtained from the set of markings of $\Sigma$ by removing any half-open interval $J \subsetneq I$ which contains $y$ but not $x$ and such that $I \setminus J$ is connected, cf.~\Cref{fig: cutting}. Choose now a finitely-full formal graded arc system $\cA$ of $\Sigma^{(1)}$. Note that, there is a unique arc, say $\gamma_U \in \cA$ which lies in in $U$. Equivalently, we may regard $\cA$ as a finitely-full and formal arc system on the original surface $\Sigma$.  Let $A^{(1)}$ (resp.~$A'$) denote the proper graded gentle algebra associated to the arc system $\cA^{(1)}\coloneqq \cA \setminus \{\gamma_U\}$ (resp.~$\cA$). The important point is that, by construction, there are no flows from $\gamma_U$ to any of the arcs in $\cA \setminus \{\gamma_U\}$ and hence we obtain a semiorthogonal decomposition of the form $$\per(A')=\langle \per( A^{(1)}), \per(k[x_n]/(x_n^2))  \rangle,$$
where $k[x_n]/(x_n^2)$ is the graded gentle algebra associated to $\gamma_U$. Furthermore, by \Cref{fukmorita}, $A$ and $A'$ are Morita equivalent, so that $\per(A)$ has the same semiorthogonal decomposition. Now,  we simply iterate the previous constructions, starting with the connected component of $\Sigma^{(1)}$ which contains the arc system $\cA^{(1)}$ corresponding to $A^{(1)}$ and finding a semiorthogonal decomposition of $\per(A^{(1)})$ of the same shape. Induction over the number of fully marked components together with \Cref{prop: Fukaya category properness  smoothness} then eventually yields a choice of a graded finitely-full formal arc system on $\Sigma$ whose associated graded proper gentle algebra $B$ is Morita equivalent to $A$ and such that $\per(B)$ admits the desired semiorthogonal decomposition with $A^{(n)}$ and each algebra $k[x_i]/(x_i^2)$ corresponding to an idempotent subalgebra of $B$.
\end{proof}
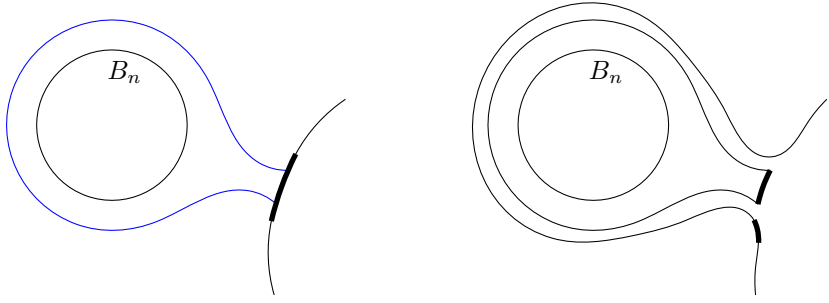
\begin{figure}
\begin{displaymath}
\begin{tikzpicture}[scale=0.8]
    	\begin{scope}[shift={(0,0)}, rotate={-20}]

        \draw[hobby, blue] plot [] coordinates {({-0.6*cos((1.25 * 120))},{0.6*sin((1.25*120))}) (0, 0.25) (-1.25,1.25) (-2.5,1.75) (-4.25,0) (-2.5,-1.75) (-1.25,-1.25) (0, -0.25) ({-0.6*cos((1.75 * 120))},{0.6*sin((1.75*120))})}; 
	\draw (-2.5,0) circle (1.25); 
    
	\draw[hobby, black] ({-2*cos((1 * 120))},{2*sin((1*120))})  to[curve through={({-0.8*cos((1.1 * 120))},{0.8*sin((1.1*120))}) (0.5,0) ({-0.8*cos((1.9 * 120))},{0.8*sin((1.9*120))}) }]({-2*cos((2 * 120)) },{2*sin((2*120))});

 \draw[hobby, line width=2pt, black] plot coordinates { ({-0.8*cos((1.1 * 120))},{0.8*sin((1.1*120))}) (0.5,0) ({-0.8*cos((1.9 * 120))},{0.8*sin((1.9*120))}) }; 

       \node at (-2.6,0.9) {$B_n$};
    \end{scope}

    \begin{scope}[shift={(8,0)}, rotate={-20}]
	\draw (-2.5,0) circle (1.25); 

 \draw[hobby, line width=2pt, black] plot coordinates { ({-0.6*cos((1.25 * 120))},{0.6*sin((1.25*120))}) (0.5,0) ({-0.6*cos((1.75 * 120))},{0.6*sin((1.75*120))}) }; 
   \draw[hobby, line width=2pt, black] plot coordinates {  (0.55,-0.57)  (0.65,-0.69) (0.75,-0.9)}; 


      \draw[hobby, black] plot [] coordinates {({-2*cos((1 * 120)) },{2*sin((1*120))}) ({-1.5*cos((1.03 * 120))},{1.5*sin((1.03*120))})  ({-.7*cos((1.1* 120))},{.7*sin((1.1*120))}) (-0.5,0.95)  (-1.25,1.5) (-2.5,2) (-4.5,0) (-2.5,-2) (-1.25,-1.5) (-0.5,-0.95)  ({-.7*cos((1.9* 120))},{.7*sin((1.9*120))})  ({-1.5*cos((1.97* 120))},{1.5*sin((1.97*120))})  ({-2*cos((2 * 120)) },{2*sin((2*120))})  }; 

 \draw[hobby, black] plot [] coordinates {({-0.6*cos((1.25 * 120))},{0.6*sin((1.25*120))}) (0, 0.25) (-1.25,1.25) (-2.5,1.75) (-4.25,0) (-2.5,-1.75) (-1.25,-1.25) (0, -0.25) ({-0.6*cos((1.75 * 120))},{0.6*sin((1.75*120))})}; 
     
       \node at (-2.6,0.9) {$B_n$};
    \end{scope}





\end{tikzpicture}
\end{displaymath}
\caption{The arc $\delta_n$ on the left (blue) and the two connected components of $\Sigma^{(1)}$ on the right. The connected component on the right containing $B_n$ gives rise to the algebra $k[x_n]/(x_n^2)$.} \label{fig: cutting}
\end{figure}

\begin{lem}\label{dualnumberlem}
    Let $A$ be the dg algebra $k[x]/(x^2)$ with $|x|=n \in \mathbb{Z}$. Then $A$ is reflexive and the unique simple $A$-module is a $\pvd$-generator.
\end{lem}
\begin{proof}It is easy to check that $A$ is derived complete at $k$. Hence for $n<0$, reflexivity follows from \Cref{lem: coconnective semisimple H^0 = reflexive} and the generation statement from \Cref{prop: keller-nicolas locality}. For $n\geq 0$, reflexivity follows from \Cref{connectiveDC} and the generation statement from \Cref{prop: Dfd generator connective case}.
\end{proof}

\begin{thm}\label{thm: proper graded gentle is reflexive}
For every proper graded gentle algebra $A$, $\per(A)$ and $\pvd{(A)}$ are reflexive dg categories. {In particular, if $\Sigma$ is a graded marked surface which has at least one marked interval, then $\Fuk^f(\Sigma)$ is reflexive.}
\end{thm}
\begin{proof}
  This follows from the repeated application of \Cref{thm:Gluing over Semiorthogonal decompositions} to \Cref{prop: semiorthogonal decomposition gentle}, combined with the fact that smooth proper dg categories are reflexive (\Cref{refexamples}) and the dual numbers are reflexive (\Cref{dualnumberlem}).
\end{proof}

The decomposition of \Cref{prop: semiorthogonal decomposition gentle} can also be used to show that the maximal semisimple quotient of any proper graded gentle algebra is a $\pvd$-generator:

\begin{thm}
Let $A$ be a proper graded gentle algebra. Then $A/\rad(A)$ is a $\pvd$-generator for $A$. In other words, the thick subcategory of $\cD(A)$ generated by the simple $A$-modules coincides with $\pvd(A)$.
\end{thm}
\begin{proof} We know from \Cref{prop: semiorthogonal decomposition gentle} and its proof that $A$ is Morita equivalent to a proper graded gentle algebra $B$ which admits a semiorthogonal decomposition of the form
$$\per(B) \simeq \langle \per(B'), \per(k[x_1]/(x_1^2)), \dots, \per(k[x_n]/(x_n^2)) \rangle,$$
so that $B'$ and all graded dual numbers correspond to idempotent subalgebras of $B$. Moreover, $B'$ is a smooth proper graded gentle algebra. We claim that this is sufficient to conclude that $\pvd(B)$ is generated by $B/\rad(B)$. Indeed, the above semiorthogonal decomposition induces a semiorthogonal decomposition of $\pvd(B)$ (with orders reversed). On the other hand, by \Cref{dualnumberlem}, $\pvd(k[x_i]/(x_i^2))$ is generated by the simple $k[x_i]/(x_i^2)$-module and, as a thick subcategory of $\pvd(B)$ is generated by the simple $A$-module corresponding to the idempotent of $A$ associated with the subalgebra $k[x_i]/(x_i^2)$. Because $B'$ is smooth and proper (and hence reflexive), $\pvd(B')\simeq \per(B')$ is generated by the simple $B'$-modules. As before, these can equally be considered as simple $A$-modules which generate the thick subcategory of $\pvd(B)$ corresponding to $\pvd(B')$. Altogether, this shows that $B/\rad(B)$ generates $\pvd(B)$. 

Since $B$ is reflexive by \Cref{thm: proper graded gentle is reflexive} and derived complete (it is proper), is thus follows from \Cref{lem: reflexive and Koszul dual} that $B^!$ is reflexive and $\thick_{B^!} \left(B/\rad(B)\right)=\pvd(B^!)$. Hence, in order to show that $A/\rad(A)$ is a generator of $\pvd(A)$, it suffices by \Cref{lem: reflexive and Koszul dual} to show that there is an equivalence $\per(A^!) \simeq \per(B^!)$ which identifies $\thick_{A^!} A/\rad(A)$ with $\thick_{B^!} B/\rad(B)$. We recall from \Cref{prop: Koszul dual of Fukaya category}, that, up to passing to a category algebra, $A$ and $B$ are the Koszul duals of $\cF_A=\cF(\Gamma_A)$ and $\cF_B=\cF(\Gamma_B)$ for full formal arc systems $\Gamma_A, \Gamma_B \subseteq \Sigma$ on a graded marked surface $\Sigma$. In other words, $A^!$ and $B^!$ are the derived completions of $\cF_A$ and $\cF_B$ at the thick subcategories $\cB_A$ and $\cB_B$ generated by the respective canonical augmentations. But now we recall from \cite[Appendix C, Step 6 (3)]{OpperPlamondonSchroll} that under the equivalence $\per(\cF(\Gamma_A)) \simeq \per(\cF(\Gamma_B))$ from \Cref{fukmorita}, $\cB_A$ and $\cB_B$ are mapped to each other. Our assertion therefore follows from the Morita invariance of derived completion, cf.~\Cref{dercomprmk}(2). 
\end{proof}
 By \Cref{cor: Fukaya category as proper graded gentle}, we also obtain the following:
\begin{cor} If $\Sigma$ is a graded marked surface with at least one marked interval and without fully marked components of vanishing winding number, then $\Fuk(\Sigma)$ is reflexive. 
\end{cor}
\begin{ex}\label{exa: non reflexive Fukaya}
One cannot drop the assumption on the winding numbers. Indeed, the polynomial ring $k[x]$ with $x$ in degree zero is not reflexive by \Cref{noeththm}. The associated surface is an annulus with one marked interval and one fully marked component whose winding numbers vanish.
\end{ex}

\begin{footnotesize}
	\bibliographystyle{alpha}
	\bibliography{references.bib}
\end{footnotesize}

\end{document}